\documentclass[a4paper]{amsart}
\usepackage{hyperref,enumerate}
\usepackage{amsmath,amssymb,microtype}
\usepackage{amssymb}
\usepackage{todonotes}
\usepackage{xparse}

\title{Knapsack in hyperbolic groups}
\author{Markus Lohrey}
\address{Universit\"at Siegen, Germany}
\email{lohrey@eti.uni-siegen.de}
\thanks{This work has been supported by the DFG research project
LO 748/13-1.}

\theoremstyle{plain}
\newtheorem{theorem}{Theorem}[section]
\newtheorem{proposition}[theorem]{Proposition}
\newtheorem{lemma}[theorem]{Lemma}
\newtheorem{corollary}[theorem]{Corollary}

\theoremstyle{definition}

\newcommand{\NP}{\mathsf{NP}} 
\newcommand{\NL}{\mathsf{NL}} 
\newcommand{\N}{\mathbb{N}}
\newcommand{\Z}{\mathbb{Z}}
\newcommand{\DHB}{H_3(\Z)}
\newcommand{\shlex}{\mathsf{shlex}}
\newcommand{\Sol}{\mathsf{sol}}
 \newcommand{\LogCFL}{\ensuremath{\mathsf{LogCFL}}}

\begin{document}

\maketitle

\begin{abstract}
Recently knapsack problems have been generalized from the integers to arbitrary finitely generated groups.
The knapsack problem for a finitely generated group $G$ is the following decision problem: given a tuple 
$(g, g_1, \ldots, g_k)$ of elements of $G$, are there natural numbers $n_1, \ldots, n_k \in \N$ such that $g = g_1^{n_1} \cdots g_k^{n_k}$
holds in $G$? Myasnikov, Nikolaev, and Ushakov proved that for every (Gromov-)hyperbolic group, the knapsack problem can
be solved in polynomial time. In this paper, the precise complexity of the knapsack problem for hyperbolic group is determined:
for every hyperbolic group $G$, the knapsack problem belongs to the complexity class $\LogCFL$,
and it is $\LogCFL$-complete if $G$ contains a free group of rank two.
Moreover, it is shown that for every hyperbolic group $G$ and  every tuple $(g, g_1, \ldots, g_k)$ of elements of $G$ the set of all $(n_1, \ldots, n_k) \in \N^k$
such that $g = g_1^{n_1} \cdots g_k^{n_k}$ in $G$ is semilinear and a semilinear 
representation where all integers are of size polynomial in the total 
geodesic length of the $g, g_1, \ldots, g_k$ can be computed. Groups with this property are also called knapsack-tame.
This enables us to show that knapsack can be solved in $\LogCFL$ for every group that belongs to the closure of 
hyperbolic groups under free products and direct products with $\Z$.
\end{abstract}

\section{Introduction}

In \cite{MyNiUs14}, Myasnikov, Nikolaev, and Ushakov initiated the investigation of  
discrete optimization problems, which are  usually formulated over the integers,
for arbitrary (possibly non-commutative) groups. One of these problems is the 
{\em knapsack problem} for a finitely generated group $G$:
The input is a sequence of group elements $g_1, \ldots, g_k, g \in G$ (specified
by finite words over the generators of $G$) and it is asked whether there exists a tuple
$(n_1, \ldots, n_k) \in \mathbb{N}^k$
such that $g_1^{n_1} \cdots g_k^{n_k} = g$ in $G$. 
For the particular case $G = \mathbb{Z}$  (where the additive notation 
$n_1 \cdot g_1 + \cdots + n_k \cdot g_k = g$ is usually preferred)
this problem is {\sf NP}-complete (resp., $\mathsf{TC}^0$-complete)
if the numbers $g_1, \ldots, g_k,g \in \Z$ are 
encoded in binary representation \cite{Karp72,Haa11} (resp., unary notation
\cite{ElberfeldJT11}).

In \cite{MyNiUs14}, Myasnikov et al.~encode elements of the finitely generated group $G$ by words over the group generators
and their inverses, which corresponds to the unary encoding of integers. There is also an encoding of words 
that corresponds to the binary encoding of integers, so called straight-line programs, and knapsack problems under this
encoding have been studied in \cite{LohreyZetzsche2016a}. In this paper, we only consider the case where input words are explicitly represented.
Here is a list of known results concerning the knapsack  problem:
\begin{itemize}
\item Knapsack can be solved in polynomial
time for every hyperbolic group \cite{MyNiUs14}. In \cite{FrenkelNU15} this result was extended to free products of any
finite number of hyperbolic groups and finitely generated abelian groups.
\item There are nilpotent groups of class $2$ for which knapsack is undecidable. Examples
are direct products of sufficiently many copies of the discrete Heisenberg group $H_3(\mathbb{Z})$ \cite{KoenigLohreyZetzsche2015a},
and free nilpotent groups of class $2$ and sufficiently high rank \cite{MiTr17}.
\item Knapsack for $H_3(\mathbb{Z})$ is decidable \cite{KoenigLohreyZetzsche2015a}. In particular, 
together with the previous point it follows that decidability
of knapsack is not preserved under direct products.
\item Knapsack is decidable for every co-context-free group~\cite{KoenigLohreyZetzsche2015a}, i.e., groups where the set of all words over the generators that do not represent the identity
is a context-free language. Lehnert and Schweitzer \cite{LehSch07} have shown that the Higman-Thompson groups are
co-context-free.
\item Knapsack belongs to $\NP$ for all virtually special groups (finite extensions of subgroups
of graph groups) \cite{LohreyZ18}. The class of virtually special groups is very rich. It contains all Coxeter groups, 
one-relator groups with torsion, fully residually free groups, and  fundamental groups of hyperbolic 3-manifolds.
For graph groups (also known as right-angled Artin groups) a complete classification of the complexity
of knapsack was obtained in \cite{LohreyZ18}: If the underlying graph contains an induced path or cycle on 4 nodes, then knapsack
is $\NP$-complete; in all other cases knapsack can be solved in polynomial time (even in {\sf LogCFL}).
\item Decidability of knapsack is preserved under finite extensions, HNN-extensions over finite associated subgroups 
and amalgamated free products over finite subgroups \cite{LohreyZetzsche2016a}.
\end{itemize}
In this paper we further investigate the knapsack problem in hyperbolic groups. The definition of hyperbolic groups requires that
all geodesic triangles in the Cayley-graph are $\delta$-slim for a constant $\delta$; see Section~\ref{sec-hyp} for details. The class of hyperbolic groups has
several alternative characterizations (e.g., it is the class of finitely generated groups with a linear Dehn function), which
gives hyperbolic groups a prominent role in geometric group theory.  Moreover, in a certain probabilistic sense, almost all finitely presented groups are hyperbolic \cite{Gro87,Olsh92}.
 Also from a computational viewpoint, hyperbolic groups have nice properties: it is known that the word problem and the conjugacy 
 problem can be solved in linear time
 \cite{EpsteinH06,Hol00}. As mentioned above, knapsack can be solved in polynomial
time for every hyperbolic group \cite{MyNiUs14}. Our first main result of this paper provides a precise characterization of the complexity of knapsack for hyperbolic groups:
for every hyperbolic group, knapsack belongs to $\LogCFL$, which is the class of all problems that are logspace-reducible
to a context-free language. $\LogCFL$ has several alternative characterizations, see Section~\ref{sec-logcfl} for details. The $\LogCFL$ upper bound
for knapsack in hyperbolic groups improves the polynomial upper bound shown in \cite{MyNiUs14}, and also generalizes a result from \cite{Lo05ijfcs}, stating
that the word problem for a hyperbolic group is in $\LogCFL$. For hyperbolic groups that contain a copy of a non-abelian free group (such hyperbolic
groups are called non-elementary) it follows from  \cite{LohreyZ18} that knapsack is $\LogCFL$-complete. Hyperbolic groups 
that contain no copy of a non-abelian free group (so called elementary hyperbolic groups) are known to be virtually cyclic, in which case
knapsack belongs to nondeterministic logspace ($\NL$), which is contained in $\LogCFL$.

In Section~\ref{sec-hyp-ksl} we prove our second main result: for every hyperbolic group $G$ and  every tuple $(g, g_1, \ldots, g_k)$ of elements of $G$ the set of all $(n_1, \ldots, n_k) \in \N^k$
such that $g = g_1^{n_1} \cdots g_k^{n_k}$ in $G$ is effectively semilinear. In other words: the set of all solutions of a knapsack instance in $G$ is semilinear.
Groups with this property are also called knapsack-semilinear. For the special case $G = \mathbb{Z}$ this is well-known (the set of solutions of a linear equation is 
Presburger definable and hence semilinear). Clearly, knapsack is decidable for every knapsack-semilinear group (due to the effectiveness assumption). In a series of 
recent papers it turned out that the class of knapsack-semilinear groups is surprisingly rich. It contains all virtually special groups~\cite{DBLP:journals/corr/LohreyZ15}
and all co-context-free group~\cite{KoenigLohreyZetzsche2015a} and is closed under the following constructions:
\begin{itemize}
\item going to a finitely generated subgroup (this is trivial) and going to a finite group extension \cite{LohreyZetzsche2016a},
\item HNN-extensions over finite associated subgroups and amalgamated free products over finite subgroups \cite{LohreyZetzsche2016a},
\item direct products (in contrast, the class of groups with a decidable knapsack problem is not closed under direct products), 
\item restricted wreath products \cite{GanardiKLZ18}.
\end{itemize}
Our proof of the knapsack-semilinearity of a hyperbolic group shows an additional quantitative statement: If the group  elements
$g, g_1, \ldots, g_k$ are represented by words over the generators and the total length of these words is $N$, then the set 
$\{ (n_1, \ldots, n_k) \in \N^k \mid g = g_1^{n_1} \cdots g_k^{n_k} \text{ in } G\}$ has a semilinear representation, where all vectors
only contain integers of size at most $p(N)$. Here, $p(x)$ is a fixed polynomial that only depends on $G$. Groups with this property are called knapsack-tame
in \cite{LohreyZ18}. In \cite{LohreyZ18}, it is shown that the class of knapsack-tame groups is closed under free products and 
direct products with $\mathbb{Z}$. Using this, we can 
show in Section~\ref{sec-more-logcfl} that knapsack can be solved in $\LogCFL$ for every group that belongs to the closure of 
hyperbolic groups under free products and direct products with $\Z$.

Recently, it was shown that the compressed version of the knapsack problem, 
where input words are encoded by straight-line programs,
is $\NP$-complete for every infinite hyperbolic group \cite{HoLoSchl19}.

\section{General notations}

We assume that the reader is familiar with basic concepts from group theory and formal languages. 
The empty word is denoted with $\varepsilon$.
For a word $w = a_1 a_2 \cdots a_n$ let $|w|=n$ be the length of $w$,
and for $1 \leq i \leq j \leq n$ let $w[i] = a_i$, $w[i:j] = a_i \cdots a_j$,
$w[:i] = w[1:i]$ and $w[i:] = w[i:n]$.
Moreover, let $w[i:j] = \varepsilon$ for $i > j$.

A set of vectors $A \subseteq \mathbb{N}^k$ is {\em linear} if there exist vectors
$v_0, \ldots, v_n \in \mathbb{N}^k$ such that 
$A = \{ v_0 + \lambda_1 \cdot v_1 + \cdots +  \lambda_n \cdot v_n \mid \lambda_1,\ldots,\lambda_n\in\N\}$.
The tuple of vectors $(v_0, \ldots, v_n)$ is a \emph{linear represention} of $A$.
Its {\em magnitude} is the largest number appearing in one the vectors $v_0, \ldots, v_n$.
A set $A \subseteq \mathbb{N}^k$ is {\em semilinear} if it is a finite union of linear sets $A_1, \ldots, A_m$.
A {\em semilinear representation} of $A$ is a list of linear representations for the linear sets $A_1, \ldots, A_m$.
Its {\em magnitude} is the maximal magnitude of the linear representations for the sets  $A_1, \ldots, A_m$.
The magnitude of a semilinear set $A$ is the smallest magnitude among all semilinear representations of $A$.

In the context of knapsack problems, we will consider 
semilinear sets as  sets of mappings $f : \{x_1, \ldots, x_k\} \to \N$ for a finite set of variables
$X = \{x_1, \ldots, x_k\}$. Such a mapping $f$ can be identified with the vector $(f(x_1), \ldots, f(x_k))$.
This allows to use all vector operations (e.g. addition and scalar multiplication) on the set
$\N^X$ of all mappings from $X$ to $\N$. The pointwise product $f \cdot g$ of two mappings
$f,g \in \N^X$ is defined by $(f \cdot g)(x) = f(x) \cdot g(x)$ for all $x \in X$. Moreover, for mappings
$f \in \N^X$, $g \in \N^Y$ with $X \cap Y = \emptyset$ we define $f \oplus g : X \cup Y \to \N$
by $(f \oplus g)(x) = f(x)$ for $x \in X$ and  $(f \oplus g)(y) = g(y)$ for $y \in Y$. All operations on
$\N^X$ will be extended to subsets of $\N^X$ in the standard pointwise way.

It is well-known that the semilinear subsets of $\N^k$ 
are exactly the sets definable in {\em Presburger arithmetic}. These are 
those sets that can be defined with a first-order formula $\varphi(x_1, \ldots, x_k)$ over the structure
$(\N,0,+,\le)$~\cite{GinsburgSpanier1966}. Moreover, the transformations between such a first-order formula and an equivalent  semilinear  representation
are effective. In particular, the semilinear sets are effectively closed under Boolean operations.

\section{Hyperbolic groups} \label{sec-hyp}

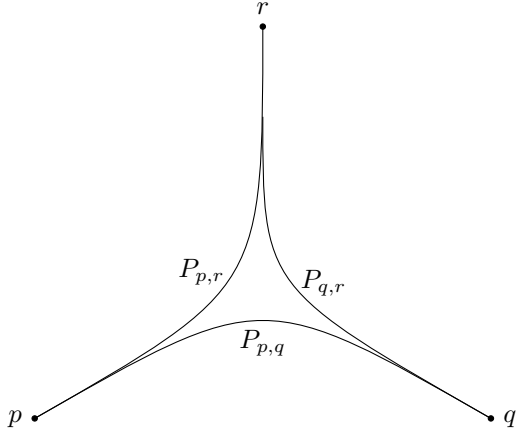
\begin{figure}[t]
 \centering{
 \scalebox{1}{
\begin{tikzpicture}
\tikzstyle{small} = [circle,draw=black,fill=black,inner sep=.25mm]
\node (p) at (0,0) [small, label=left:$p$] {};
\node (q) at (6,0) [small, label=right:$q$] {};
\node (r) at (3,5.19615) [small, label=above:$r$] {};
\draw (0,0) .. controls (3,1.73205)  .. node[pos=.5,below=0mm] {$P_{p,q}$} (6,0);
\draw (0,0) .. controls (3,1.73205)  .. node[pos=.5,above=0mm, left=0mm] {$P_{p,r}$} (3,5.19615);
\draw (3,4) .. controls (3,1.73205)  .. node[pos=.5,above=0mm, right=0mm] {$P_{q,r}$} (6,0);
\end{tikzpicture}
}}
\caption{\label{fig-geo-tri}The shape of a geodesic triangle in a hyperbolic group}
\end{figure}

Let $G$ be a finitely generated group with the finite symmetric generating set $\Sigma$, i.e., $a \in \Sigma$ implies that 
$a^{-1} \in \Sigma$.
The  {\em Cayley-graph} of $G$ (with respect to $\Sigma$) is the undirected  graph $\Gamma = \Gamma(G)$ with node set
$G$ and all edges $(g,ga)$ for $g \in G$ and $a \in \Sigma$. We view $\Gamma$ as a geodesic metric space,
where every edge $(g,ga)$ is identified with a unit-length interval. It is convenient to label the directed edge
from $g$ to $ga$ with the generator $a$.
The distance between two points $p,q$ is denoted with $d_\Gamma(p,q)$.
For $g \in G$ let $|g| = d_\Gamma(1,g)$.
For $r \geq 0$,  let
$\mathcal{B}_r(1) = \{ g \in G \mid d_\Gamma(1,g) \leq r \}$.

Paths can be defined in a very general way for metric spaces, but we only need paths
that are induced by words over $\Sigma$.
Given a word $w \in \Sigma^*$ of length $n$, one obtains a unique path $P[w] : [0,n] \to \Gamma$,
which is a continuous mapping from the real interval $[0,n]$ to $\Gamma$.
It maps the subinterval $[i,i+1] \subseteq [0,n]$  isometrically 
onto the edge $(g_i, g_{i+1})$ of $\Gamma$, where $g_i$ (resp., $g_{i+1})$ is the group element
represented by the word $w[:i]$ (resp., $w[:i+1]$).
The path $P[w]$ starts in $1 = g_0$ and ends in $g_n$ (the group element represented by $w$).
We also say that $P[w]$ is the unique path that starts in $1$ and is labelled with the word $w$.
More generally, for $g \in G$ we denote
with $g \cdot P[w]$ the path that starts in $g$  and is labelled with $w$.
When writing $u \cdot P[w]$ for a word $u \in \Sigma^*$, we mean the path
$g \cdot P[w]$, where $g$ is the group element represented by $u$.
A path  $P : [0,n] \to \Gamma$ of the above form is geodesic if $d_\Gamma(P(0),P(n)) = n$; it is
a $(\lambda,\epsilon)$-{\em quasigeodesic} if for all points
$p = P(a)$ and $q = P(b)$ we have $|a-b| \leq \lambda \cdot d_\Gamma(p,q) + \varepsilon$;
and it is $\zeta$-{\em local} $(\lambda,\epsilon)$-{\em quasigeodesic} if for all points
$p = P(a)$ and $q = P(b)$ with $|a-b| \le \zeta$ we have $|a-b| \leq \lambda \cdot d_\Gamma(p,q) + \varepsilon$.

A word $w \in \Sigma^*$ is geodesic if the path $P[w]$ is geodesic, which means
that there is no shorter word representing the same group element from $G$.
Similarly, we define the notion of ($\zeta$-local) $(\lambda,\epsilon)$-quasigeodesic words.
A word $w \in \Sigma^*$ is {\em shortlex reduced} if it is the length-lexicographically smallest word 
that represents the same group element as $w$. For this, we have to fix an arbitrary linear
order on $\Sigma$. 
Note that if $u = xy$ is shortlex reduced then $x$ and $y$ are shortlex reduced  too.
For a word $u \in \Sigma^*$ we denote
with $\shlex(u)$ the unique shortlex reduced word that represents the same group element as $u$.

A {\em geodesic triangle} consists of three points $p,q,r \in G$ and geodesic paths $P_1 = P_{p,q}$, $P_2 = P_{p,r}$, $P_3 = P_{q,r}$
(the three sides of the triangle),
where $P_{x,y}$ is a geodesic path from $x$ to $y$. We call a geodesic triangle {\em $\delta$-slim}
for $\delta \geq 0$, if for all $i \in \{1,2,3\}$, every point on $P_i$ has distance at most 
$\delta$ from a point on $P_j \cup P_k$, where $\{j,k\} = \{1,2,3\} \setminus \{i\}$.
The group $G$ is called  {\em $\delta$-hyperbolic}, if  every geodesic triangle is 
$\delta$-slim. Finally, $G$ is  hyperbolic, if it is  $\delta$-hyperbolic
for some $\delta \geq 0$. Figure~\ref{fig-geo-tri} shows the shape of a geodesic triangle in a hyperbolic group.
Finitely generated free groups are for instance $0$-hyperbolic.
The property of being hyperbolic is independent of the chosen generating set $\Sigma$. 
The word problem for every hyperbolic group can be decided in real time \cite{Hol00}. 

Let us fix a $\delta$-hyperbolic group $G$ with the finite symmetric generating set $\Sigma$
for the rest of the section, and let $\Gamma$ be the corresponding geodesic metric space.
We will apply a couple of well-known results for hyperbolic groups.

\begin{lemma}[c.f.~\mbox{\cite[8.21]{ghys1990groupes}}] \label{lemma-cyclic-words-quasi-geo}
Let $g \in G$ be of infinite order and let $n \geq 0$. 
Let $u$ be a geodesic word representing $g$.
Then the word  $u^n$ is $(\lambda,\epsilon)$-quasigeodesic, where 
$\lambda = N|g|$, $\epsilon = 2N^2 |g|^2  + 2N |g|$ and $N = |\mathcal{B}_{2\delta}(1)|$. 
\end{lemma}
Consider two paths $P_1 : [0,n_1] \to \Gamma$, $P_2 : [0,n_2] \to \Gamma$ and let $K$ be a positive
real number.
We say that $P_1$ and $P_2$ {\em asynchronously $K$-fellow travel} if there exist two continuous non-decreasing mappings
$\varphi_1 : [0,1] \to [0,n_1]$ and $\varphi_2 : [0,1] \to [0,n_2]$ such that $\varphi_1(0) = \varphi_2(0) = 0$, $\varphi_1(1) = n_1$,
$\varphi_2(1) = n_2$ and for all $0 \leq t \leq 1$, $d_\Gamma(P_1(\varphi_1(t)), P_2(\varphi_2(t))) \leq K$.
Intuitively, this means that one can travel along the paths $P_1$ and $P_2$ asynchronously with variable speeds such that
at any time instant the current points have distance at most $K$. By slightly increasing $K$ one obtains a ladder graph of 
the form shown in Figure~\ref{fig-ladder}, where the edges connecting the horizontal $P_1$- and $P_2$-labelled 
paths represent paths of length at most $K$ that connect elements from $G$.

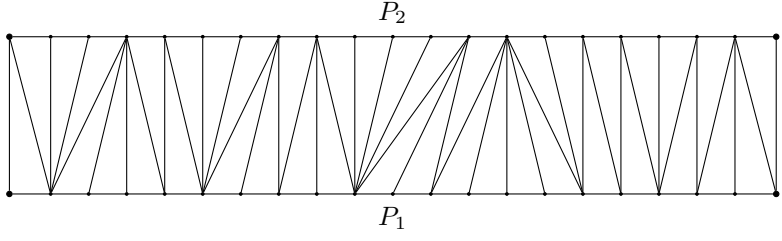
\begin{figure}[t]
 \centering{
 \scalebox{1}{
\begin{tikzpicture}
  \tikzstyle{small} = [circle,draw=black,fill=black,inner sep=.25mm]
  \tikzstyle{tiny} = [circle,draw=black,fill=black,inner sep=.1mm]
  \tikzstyle{zero} = [circle,inner sep=0mm]

    \node[small] (a1) {} ;
    \node[small,  right = 10cm of a1] (a2) {};
    \node[small,  above = 2cm of a1] (b1) {};
    \node[small,  right = 10cm of b1] (b2) {};
    
    \node[zero,  below = 2.5mm of a1] (a1') {} ;
    \node[zero,  below = 2.5mm of a2] (a2') {} ;
    \node[zero,  above = 2.5mm of b1] (b1') {} ;
    \node[zero,  above = 2.5mm of b2] (b2') {} ;
    
       \draw [-] (a1) to   
      node[pos=.5,below=.5mm]  {$P_1$}   
      node[pos=.05, tiny] (1) {}  
      node[pos=.1, tiny] (2) {}  
      node[pos=.15, tiny] (3) {} 
      node[pos=.2, tiny] (4) {}  
      node[pos=.25, tiny] (5) {}  
      node[pos=.3, tiny] (6) {}  
      node[pos=.35, tiny] (7) {}  
      node[pos=.4, tiny] (8) {}  
      node[pos=.45, tiny] (9) {}  
      node[pos=.5, tiny] (10) {}  
      node[pos=.55, tiny] (11) {}  
      node[pos=.6, tiny] (12) {}  
      node[pos=.65, tiny] (13) {}  
      node[pos=.7, tiny] (14) {}  
      node[pos=.75, tiny] (15) {}  
      node[pos=.8, tiny] (16) {}  
      node[pos=.85, tiny] (17) {} 
      node[pos=.9, tiny] (18) {} 
      node[pos=.95, tiny] (19) {}  
        (a2);
        
        \draw [-] (b1) to   
      node[pos=.5,above=.5mm]  {$P_2$}   
      node[pos=.05, tiny] (1') {}  
      node[pos=.1, tiny] (2') {}  
      node[pos=.15, tiny] (3') {} 
      node[pos=.2, tiny] (4') {}  
      node[pos=.25, tiny] (5') {}  
      node[pos=.3, tiny] (6') {}  
      node[pos=.35, tiny] (7') {}  
      node[pos=.4, tiny] (8') {}  
      node[pos=.45, tiny] (9') {}  
      node[pos=.5, tiny] (10') {}  
      node[pos=.55, tiny] (11') {}  
      node[pos=.6, tiny] (12') {}  
      node[pos=.65, tiny] (13') {}  
      node[pos=.7, tiny] (14') {}  
      node[pos=.75, tiny] (15') {}  
      node[pos=.8, tiny] (16') {}  
      node[pos=.85, tiny] (17') {} 
      node[pos=.9, tiny] (18') {} 
      node[pos=.95, tiny] (19') {}  
        (b2);

       \draw [-] (a1) to  (b1);
      \draw [-] (1) to  (b1);               
      \draw [-] (1) to (1');
      \draw [-] (1) to  (2');
      \draw [-] (1) to  (3');
      \draw [-] (2) to  (3');
      \draw [-] (3) to  (3');
      \draw [-] (4) to  (3');
      \draw [-] (4) to  (4');
      \draw [-] (5) to  (4');
      \draw [-] (5) to  (5');    
      \draw [-] (5) to  (6');
      \draw [-] (5) to  (7');
      \draw [-] (6) to  (7');
      \draw [-] (7) to  (7');
      \draw [-] (7) to  (8');
      \draw [-] (8) to  (8');
      \draw [-] (9) to  (8');
      \draw [-] (9) to  (9');
      \draw [-] (9) to  (10');
      \draw [-] (9) to  (11');
      \draw [-] (9) to  (12');
       \draw [-] (10) to  (12');
       \draw [-] (11) to  (12');
       \draw [-] (11) to  (13');
       \draw [-] (12) to  (13');
       \draw [-] (13) to  (13');
       \draw [-] (14) to  (13');
       \draw [-] (15) to  (13');
       \draw [-] (15) to  (14');
       \draw [-] (15) to  (15');
       \draw [-] (16) to  (15');
       \draw [-] (16) to  (16');
       \draw [-] (17) to  (16');
       \draw [-] (17) to  (17');
       \draw [-] (17) to  (18');
       \draw [-] (18) to  (18');
       \draw [-] (18) to  (19');
       \draw [-] (19) to  (19');
        \draw [-] (a2) to  (19');
    \draw [-] (a2) to  (b2);
   \end{tikzpicture}}}
\caption{\label{fig-ladder} Paths that asynchronously $K$-fellow travel}
  \end{figure}

\begin{lemma}[c.f.~\cite{MyNi14}] \label{lemma-asynch-fellow-travel}
Let $P_1$ and $P_2$ be $(\lambda,\epsilon)$-quasigeodesic paths in $\Gamma_G$ and assume that
$P_i$ starts in $g_i$ and ends in $h_i$. Assume that $d_\Gamma(g_1,g_2), d_\Gamma(h_1,h_2) \leq h$.
Then there exists a computable bound 
$K = K(\delta, \lambda,\epsilon, h) \geq h$ such that $P_1$ and $P_2$ 
asynchronously $K$-fellow travel.
\end{lemma}
Finally we need the following lemma for splitting quasigeodesic rectangles:  
  
  \begin{lemma} \label{lemma-rectangle}
 Fix constants $\lambda,\epsilon$ and let $\kappa = K(\delta,\lambda,\epsilon,0)$ be taken from
 Lemma~\ref{lemma-asynch-fellow-travel}. 
Let $v_1, v_2 \in \Sigma^*$ be geodesic words 
and $u_1, u_2 \in \Sigma^*$ $(\lambda,\epsilon)$-quasigeodesic words
such that $v_1 u_1  =  u_2 v_2$ in $G$. Consider 
a factorization $u_1 = x_1y_1$ with $|x_1| \ge \lambda (|v_1| + 2 \delta + \kappa) + \epsilon$
and $|y_1| \ge \lambda (|v_2| + 2 \delta  + \kappa) + \epsilon$
Then there exists a factorization $u_2 = x_2 y_2$ and $c \in \mathcal{B}_{2\delta + 2\kappa}(1)$ such
that $v_1 x_1 = x_2 c$ and $c y_1 = y_2 v_2$ in $G$. 
\end{lemma}

\begin{proof}
The construction is shown in Figure~\ref{lemma-rectangle}.
Let $t_1, t_2, x'_1, y'_1$ be geodesic words with $t_1 = u_1$, $t_2 = u_2$,
$x_1 = x'_1$ and $y_1 = y'_1$  in $G$.
Since $u_1$ is $(\lambda,\epsilon)$-quasigeodesic,
we get $|x'_1| \geq (|x_1|-\epsilon)/\lambda \geq |v_1| + 2 \delta  + \kappa$ and
$|y'_1| \geq (|y_1|-\epsilon)/\lambda \geq |v_2| + 2 \delta  + \kappa$.
By Lemma~\ref{lemma-asynch-fellow-travel}
the paths $P[t_1]$ and $P[u_1$] asynchronously $\kappa$-fellow travel.
Hence, there exists a factorization $t_1 = r_1 s_1$ and $c_1 \in \mathcal{B}_{\kappa}(1)$
such that $r_1 c_1 = x_1 = x'_1$ and $c_1 y'_1 = c_1 y_1 = s_1$ in $G$. This implies $|r_1| \geq |x'_1|-\kappa \geq |v_1| + 2\delta$
and $|s_1| \geq |y'_1|-\kappa \geq |v_2| + 2\delta$. Consider the geodesic rectangle with the paths
$Q_1 = P[v_1]$, $P_1 = v_1 \cdot P[t_1]$, $P_2 = P[t_2]$, and $Q_2 = u_2 \cdot P[v_2]$.
Since geodesic rectangles are $2 \delta$-slim, there exists a point $p_2 \in P_2 \cup Q_1 \cup Q_2$ that
has distance at most $2\delta$ from $p_1 = P_1(|r_1|)$. By the triangle inequality we must have $p_2 \in P_2$. This yields a factorization
$t_2 = r_2 s_2$ (where $p_2 = P_2(|r_2|)$)
and $c' \in \mathcal{B}_{2\delta}(1)$  such that $v_1 r_1 = r_2 c'$ and $c' s_1 = s_2 v_2$ in $G$.  Finally, since
$P[t_2]$ and $P[u_2$] asynchronously $\kappa$-fellow travel, we obtain a factorization $u_2 = x_2 y_2$ 
and $c_2  \in \mathcal{B}_{\kappa}(1)$ such that $x_2 c_2 = r_2$ and $c_2 s_2 = y_2$ in $G$.
Let $c = c_2 c' c_1 \in \mathcal{B}_{2\delta + 2\kappa}(1)$. We get $x_2 c = x_2 c_2 c' c_1 = 
r_2 c' c_1 = v_1 r_1 c_1 = v_1 x_1$ and  $c y_1 = c_2 c' c_1 y_1 = c_2 c' s_1 = c_2 s_2 v_2 = y_2 v_2$.
\end{proof}

\begin{figure}[t]
 \centering{
 \scalebox{1}{
\begin{tikzpicture}
  \tikzstyle{small} = [circle,draw=black,fill=black,inner sep=.15mm]
  \tikzstyle{zero} = [circle,inner sep=0mm]

    \node[small] (1) {} ;
    \node[small,  right = 6cm of 1] (2) {};
    \node[small,  above = 1.5cm of 1] (3) {};
    \node[small,  right = 6cm of 3] (4) {};
          
       \draw [->] (1) to [out=10, in=-190] node[pos=.25,above=-.7mm] {$r_2$} node[pos=0.5, small] (a) {}  node[pos=.75,above=-.7mm]  {$s_2$}  (2);
      \draw [->] (1) to [out=-30, in=-150] node[pos=.25,below=-.7mm] {$x_2$} node[pos=0.5, small] (b) {}  node[pos=.75,,below=-.7mm]  {$y_2$}  (2);
      
      \draw [->] (3) to [out=30, in=-210] node[pos=.25,above=-.7mm] {$x_1$} node[pos=0.5, small] (c) {}  node[pos=.75,above=-.7mm]  {$y_1$}  (4);
      \draw [->] (3) to [out=-10, in=-170] node[pos=.25,below=-.7mm] {$r_1$} node[pos=0.5, small] (d) {}  node[pos=.75,below=-.7mm]  {$s_1$}  (4);
      
     \draw [->](1) edge node[left=-.7mm]{$v_1$} (3);
     \draw [->](2) edge node[right=-.7mm]{$v_2$} (4);
     
      \draw [->](d) edge node[left=-.7mm]{$c_1$} (c);
      \draw [->](a) edge node[left=-.7mm]{$c'$} (d);
      \draw [->](b) edge node[left=-.7mm]{$c_2$} (a);

      \draw [-](3) edge node[below=-.7mm]{$x'_1$} (c);
       \draw [-](c) edge node[below=-.7mm]{$y'_1$} (4);

     \end{tikzpicture}}}
\caption{\label{fig-rectangle} Splitting a quasigeodesic rectangle according to Lemma~\ref{lemma-rectangle}.}
  \end{figure}
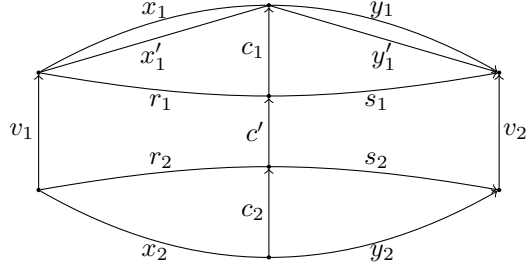

 \section{The complexity class LogCFL} \label{sec-logcfl}

The complexity class $\LogCFL$ consists of all
computational problems that are logspace reducible to a context-free language. 
The class $\LogCFL$ is included in the parallel complexity class
$\mathsf{NC}^2$ and has several alternative characterizations (see e.g.
\cite{Su78,Vol99}):
\begin{itemize}
\item logspace bounded alternating Turing-machines with polynomial tree size,
\item semi-unbounded Boolean circuits of polynomial size and logarithmic depth, and
\item logspace bounded auxiliary pushdown automata with polynomial running time.  
\end{itemize}
For our purposes, the last characterization is most suitable.
An AuxPDA (for auxiliary pushdown automaton)
is a nondeterministic pushdown automaton with a two-way input tape and 
an additional work tape. Here we only consider AuxPDAs with the following two restrictions:
\begin{itemize}
\item The length of the work tape is restricted to $O(\log n)$ for an input of length $n$
(logspace bounded).
\item There is a polynomial $p(n)$, such that every computation path
of the AuxPDA on an input of length $n$ has length at most $p(n)$
(polynomially time bounded).
\end{itemize}
Whenever we speak of an AuxPDA in the following, we implicitly assume that the AuxPDA
is logspace bounded and polynomially time bounded.
The class of languages that are accepted by AuxPDAs is exactly $\LogCFL$ \cite{Su78}.
A {\em one-way} AuxPDA is an AuxPDA that never moves the input head to the left.
Hence, in every step, the input head either does not move, or moves to the right.

For a finitely generated group $G$ with the symmetric generating set $\Sigma$ we 
define the word problem for $G$ (with respect to $\Sigma$) as the set of all words
$w \in \Sigma^*$ such that $w=1$ in $G$.
Let us say that a finitely generated group $G$ belongs to the class OW-AuxPDA if
the word problem for $G$ is recognized by a one-way AuxPDA. It is easy to see
that the latter property is independent of the generating set of $G$ (this holds, since
the class of languages recognized by one-way AuxPDAs is closed under inverse homomorphisms).

\begin{theorem} \label{thm-hyp-OW-AuxPDA}
Every hyperbolic group belongs to the class OW-AuxPDA.
\end{theorem}

\begin{proof}
Let $G$ be a hyperbolic group.
In \cite{Lo05ijfcs} it is shown that the word problem for $G$ is a growing context-sensitive language, i.e., it can
be generated by a grammar where all productions are strictly length-increasing (except for the start production $S \to \varepsilon$).
In \cite{BuOt98} it was shown that every growing context-sensitive language can 
be recognized by a one-way AuxPDA in logarithmic space and polynomial time. The result follows.
\end{proof}

\begin{theorem} \label{thm-free-directZ-OW-AuxPDA}
If the groups $G$ and $H$ belong to OW-AuxPDA then also $G*H$ and $G \times \mathbb{Z}$ belong to  OW-AuxPDA.
\end{theorem}

\begin{proof}
The proof is essentially the same as in \cite[Lemma 4.8]{LohreyZ18}, but is presented for completeness.
Let us first consider the group $G \times \mathbb{Z}$. Let $\mathcal{P}(G)$ be a one-way AuxPDA 
for the word problem of $G$.
The one-way AuxPDA $\mathcal{P}(G  \times \mathbb{Z})$ for the word problem of $G$
simulates $\mathcal{P}(G)$
on the generators of $G$. Moreover, it stores the current value of the $\mathbb{Z}$-component in binary notation
on the work tape. If the input word has length $n$, then $O(\log n)$ bits are sufficient for this. At the end,
$\mathcal{P}(G \times \mathbb{Z})$ accepts if and only if  $\mathcal{P}(G)$ accepts
and the $\mathbb{Z}$-component on the work tape is zero.

Next, we consider the group $G*H$.
We have one-way AuxPDAs $\mathcal{P}(G)$ and $\mathcal{P}(H)$ for the word problems
of $G$ and $H$, respectively. We can assume that $\mathcal{P}(G)$ (resp., $\mathcal{P}(H)$) accepts an input word
$w$ if after reading $w$ the stack is empty and $\mathcal{P}(G)$ (resp., $\mathcal{P}(H)$) is in the unique final state $q_G$ (resp., $q_H$).
This can be achieved by doing $\varepsilon$-transitions at the end of the computation. 
In the following, we call $q_G$ (resp., $q_H$) the $1$-state of $\mathcal{P}(G)$ (resp., $\mathcal{P}(H)$).

Let $\Sigma$ (resp., $\Gamma$) be the input alphabet of 
$\mathcal{P}(G)$ (resp., $\mathcal{P}(H)$), which is a symmetric generating set for $G$ (resp., $H$).
We assume that $\Sigma \cap \Gamma = \emptyset$.
Consider now an input word $w \in (\Sigma \cup \Gamma)^*$.
Let us assume that $w = u_1 v_1 u_2 v_2 \cdots u_k v_k$ with $u_i \in \Sigma^+$ and $v_i \in \Gamma^+$ (other cases
can be treated analogously). The AuxPDA $\mathcal{P}(G*H)$ starts with empty stack and simulates the AuxPDA 
$\mathcal{P}(G)$ on the prefix $u_1$. If it turns out that $u_1 = 1$ in $G$ (which means that $\mathcal{P}(G)$ is 
in its $1$-state and the stack is empty) then the AuxPDA $\mathcal{P}(G*H)$
continues with simulating $\mathcal{P}(H)$ on $v_1$. On the other hand, 
if $u_1 \neq 1$ in $G$, then $\mathcal{P}(G*H)$
pushes the state together with the work tape content of $\mathcal{P}(G)$ reached after reading $u_1$ on the stack (on top of 
the final stack content of $\mathcal{P}(G)$). This allows  $\mathcal{P}(G*H)$ to resume the computation of $\mathcal{P}(G)$ later. Then $\mathcal{P}(G*H)$
continues with simulating $\mathcal{P}(H)$ on $v_1$. 

The computation of $\mathcal{P}(G*H)$ will continue in this way. More precisely, if after reading $u_i$ (resp. $v_i$ with $i < k$) the 
AuxPDA $\mathcal{P}(G)$ (resp. $\mathcal{P}(H)$) is in its $1$-state then either
\begin{enumerate}[(i)]
\item the stack is empty or 
\item the top part of the stack is of the form $s q t$ ($t$ is the top), where $s$ is a stack content of $\mathcal{P}(H)$
(resp. $\mathcal{P}(G)$), $q$ is a state of $\mathcal{P}(H)$
(resp. $\mathcal{P}(G)$) and $t$ is a work tape content of $\mathcal{P}(H)$
(resp. $\mathcal{P}(G)$). 
\end{enumerate}
In case (i), $\mathcal{P}(G*H)$ continues with the simulation of $\mathcal{P}(H)$
(resp. $\mathcal{P}(G)$) on the word $v_i$ (resp. $u_{i+1}$) in the initial configuration.
In case (ii), $\mathcal{P}(G*H)$ continues with the simulation of $\mathcal{P}(H)$
(resp. $\mathcal{P}(G)$) on the word $v_i$ (resp. $u_{i+1}$), where the simulation is started with 
stack content $s$, state $q$, and work tape content $t$. On the other hand, if after reading $u_i$ (resp. $v_i$ with $i < k$) the 
AuxPDA $\mathcal{P}(G)$ (resp. $\mathcal{P}(H)$) is not in its $1$-state then $\mathcal{P}(G*H)$ pushes on the stack the state
and work tape content of $\mathcal{P}(G)$ reached after its simulation on $u_i$. 
This concludes the description of the AuxPDA $\mathcal{P}(G*H)$. It is a one-way AuxPDA that accepts the word 
problem of $G*H$.
\end{proof}

 \section{Knapsack problems}

Let $G$ be a finitely generated group with the finite symmetric generating set $\Sigma$.
Moreover, let $X$ be a set of formal variables that take values
from $\N$. For a subset $U\subseteq X$, we use $\N^U$ to denote the set of 
maps  $\nu \colon U \to \N$, which we call \emph{valuations}.
An \emph{exponent expression} over $G$ is a formal expression of the form $E = u_1^{x_1} v_1 u_2^{x_2} v_2  \cdots u_k^{x_k} v_k$
with $k \geq 1$ and words $u_i, v_i \in \Sigma^*$. Here, the variables do not have to be pairwise distinct. If every variable in an exponent expression occurs at most once, it is called a \emph{knapsack expression}. 
Let $X_E = \{ x_1, \ldots, x_k \}$ be the set of variables that occur in $E$.
 For a valuation $\nu\in\N^U$ such that $X_E \subseteq U$ (in which case we also say
 that $\nu$ is a valuation for $E$), we define 
 $\nu(E) = u_1^{\nu(x_1)} v_1 u_2^{\nu(x_2)} v_2  \cdots u_k^{\nu(x_k)} v_k \in \Sigma^*$.
We say that $\nu$ is a \emph{solution} of the equation $E=1$ if $\nu(E)$ evaluates to the identity element $1$ of $G$.
With $\Sol(E)$ we denote the set of all solutions $\nu \in \N^{X_E}$ of $E$. We can view $\Sol(E)$ as a subset
of $\N^k$.  The \emph{length} of $E$ is defined as $|E| = \sum_{i=1}^k |u_i|+|v_i|$, whereas $k$ is its \emph{depth}.
We define {\em solvability of exponent equations over $G$} as the following decision problem:
\begin{description}
\item[Input] A finite list of exponent expressions $E_1,\ldots,E_n$ over $G$.
\item[Question] Is $\bigcap_{i=1}^n \Sol(E_i)$ non-empty?
\end{description}
The {\em knapsack problem for $G$} is the following 
decision problem:
\begin{description}
\item[Input] A single knapsack expression $E$ over $G$.
\item[Question] Is $\Sol(E)$ non-empty?
\end{description}
It is easy to observe that the concrete choice of the generating set $\Sigma$ has no influence
on the decidability and complexity status of these problems.
Later, we will also allow exponent expressions of the form 
$v_0 u_1^{x_1} v_1 u_2^{x_2} v_2  \cdots u_k^{x_k} v_k$, which do not start with a power 
$u_1^{x_1}$. Such an exponent expression can be replaced by
$u_1^{x_1} v_1 u_2^{x_2} v_2  \cdots u_k^{x_k} v_k v_0$ without changing the set of solutions. 

The group $G$ is called {\em knapsack-semilinear} if for every knapsack expression $E$ over $G$,
the set $\Sol(E)$ is a semilinear set of vectors and a semilinear representation can be effectively computed from $E$.
Since the emptiness of the intersection of finitely many semilinear sets is decidable,  solvability of exponent equations is decidable for every
knapsack-semilinear group.
As mentioned in the introduction, the class of knapsack-semilinear groups is very rich.
An example of a group $G$, where knapsack is decidable but solvability of exponent equations
is undecidable is the Heisenberg group $\DHB$ (which consists of all upper triangular $(3 \times 3)$-matrices over the integers, where all diagonal entries
are $1$), see \cite{KoenigLohreyZetzsche2015a}. In particular, $\DHB$ is not knapsack-semilinear.

The group $G$ is called {\em polynomially knapsack-bounded} if there is a fixed polynomial $p(n)$ such that
for a given a knapsack expression $E$ over $G$, one has $\Sol(E) \neq \emptyset$ if and only if there exists
$\nu \in \Sol(E)$ with $\nu(x) \leq p(|E|)$ for all variables $x$ in $E$.

The group $G$ is called {\em knapsack-tame} if there is a fixed polynomial $p(n)$ such that
for a given a knapsack expression $E$ over $G$ one can compute a semilinear
representation for $\Sol(E)$ of magnitude at most $p(|E|)$. Thus, every knapsack-tame group is 
knapsack-semilinear as well as polynomially knapsack-bounded. The following result was shown in \cite{LohreyZ18}:

\begin{proposition}[\mbox{\cite[Proposition~4.11 and 4.17]{LohreyZ18}}] \label{prop-knapsack-tame}
If $G$ and $H$ are knapsack-tame groups then also the free product $G*H$  and the direct product $G \times \mathbb{Z}$
are knapsack-tame.
\end{proposition}

\section{Membership for acyclic automata}

An {\em acyclic NFA} is a nondeterministic
finite automaton $\mathcal{A} = (Q,\Sigma, \Delta, q_0, F)$ ($Q$ is a finite set of states, $\Sigma$ is 
the input alphabet, $\Delta \subseteq Q \times \Sigma^* \times Q$ is the set of transition triples, $q_0 \in Q$ 
is the initial state, and $F \subseteq Q$ is the set of final states) such that the relation $\{ (p,q) \in Q \times Q \mid \exists w \in \Sigma^* : (p,w,q) \in \Delta\}$
is acyclic. Note that we allow transitions labelled with words, which will be convenient in the following.

Let $G$ be a finitely generated group with the finite symmetric generating set $\Sigma$.
The {\em membership problem for acyclic NFAs over $G$} is the following computational problem:
\begin{description}
\item[Input]  an  acyclic NFA $\mathcal{A}$ with input alphabet $\Sigma$.
\item[Question]  does $\mathcal{A}$ accept a word $w \in \Sigma^*$ such that $w=1$ in $G$?
\end{description}
Again, the concrete choice of the generating set $\Sigma$ has no influence
on the decidability and complexity status of this problem.

\begin{theorem} \label{thm-acyclic-NFA}
If the group $G$ belongs to the class OW-AuxPDA, then 
membership for acyclic NFAs over $G$ belongs to $\LogCFL$.
\end{theorem}    

\begin{proof}
Let $\mathcal{P}$ be a one-way AuxPDA for the word problem of $G$.
An AuxPDA for the membership problem for acyclic NFAs over $G$
guesses a path in the acyclic input NFA $\mathcal{A}$ and thereby simulates the AuxPDA $\mathcal{P}$
on the word spelled by the guessed path.
If the final state of the input NFA $\mathcal{A}$ is reached and the AuxPDA $\mathcal{P}$ accepts at the same time,
then the overall AuxPDA accepts. It is important that the AuxPDA $\mathcal{P}$ works one-way since 
the guessed path in $\mathcal{A}$ cannot be stored in logspace. This implies that the AuxPDA cannot re-access 
the input symbols that have already been processed.  Also note that the AuxPDA 
is logspace bounded and polynomially time bounded since $\mathcal{A}$
is acyclic.
\end{proof}

\begin{theorem} \label{thm-membership-to-knapsack}
Let $G$ be a polynomially knapsack-bounded group. Then there is a logspace 
reduction from the knapsack problem for $G$ to 
membership for acyclic NFAs over $G$.
\end{theorem}    

\begin{proof}
Let $G$ be a polynomially knapsack-bounded group with the symmetric generating set $\Sigma$. We present a logspace reduction from knapsack for $G$ 
to the membership problem for acyclic NFAs. 
Consider a knapsack expression $E = u_1^{x_1} v_1 u_2^{x_2} v_2 \cdots u_{k}^{x_{k}} v_{k}$
over $G$. Since $G$ is polynomially knapsack-bounded, there exists a polynomial $p(x)$ such that $\Sol(E) \neq \emptyset$ if and only 
if there exists a solution $\nu \in \Sol(E)$ 
such that $\nu(x_i) \leq p(|E|)$ for all $1 \leq i \leq k$. We now construct an NFA $\mathcal{A}$ 
as follows: It has the state set $Q=[1,k+1]\times
[0,p(n)]$ and the following transitions. For each $i\in[1,k]$ and $j\in[0,p(n)-1]$, there are
two transitions from $(i,j)$ to $(i,j+1)$; one labeled by $u_i$ and one labeled
by $\varepsilon$. Furthermore, there is a transition from $(i,p(n))$ to
$(i+1,0)$ labeled $v_i$ for each $i\in[1,k]$. The initial state is $(1,0)$
and the unique final state is $(k+1,0)$.  

It is clear that $\mathcal{A}$ accepts a word that represents $1$ if and only if 
$\Sol(E) \neq \emptyset$.  Finally, the NFA can be clearly 
computed in logarithmic space from $E$. 
\end{proof}

\section{Complexity of knapsack in hyperbolic groups}

In this section we consider the complexity of the knapsack problem for a hyperbolic group.
In \cite{MyNiUs14} it was shown that for every hyperbolic group, knapsack can be solved
in polynomial time. Here, we improve the complexity to $\LogCFL$. We need one more result
from \cite{MyNiUs14}:

\begin{theorem}[c.f.~\mbox{\cite{MyNiUs14}}] \label{thm-ushakov}
Every hyperbolic group is polynomially knapsack-bounded.
\end{theorem}    
This result is also a direct corollary of Theorem~\ref{thm-hyperbolic-semilinear}
from the next section, stating that every hyperbolic group is knapsack-tame.

We can now easily derive the following two results:

\begin{corollary} \label{coro-acyclic-NFA-hyp}
Membership for acyclic NFAs over a hyperbolic group belongs to $\LogCFL$.
\end{corollary}    

\begin{proof}
This follows from Theorem~\ref{thm-hyp-OW-AuxPDA} and~\ref{thm-acyclic-NFA}.
\end{proof}

\begin{corollary} \label{thm-knapsack-logcfl}
For every hyperbolic groups $G$, knapsack can be solved in $\LogCFL$. Moreover, if $G$ contains a copy of $F_2$ (the free group of rank $2$)
then knapsack for $G$ is $\LogCFL$-complete.
\end{corollary}

\begin{proof}
The first statement follows from Theorems~\ref{thm-membership-to-knapsack} and~\ref{thm-ushakov} and
Corollary~\ref{coro-acyclic-NFA-hyp}.
The second statement follows from \cite[Proposition 4.26]{LohreyZ18}, where it was shown that 
knapsack for $F_2$ is $\LogCFL$-complete.
\end{proof}

\section{Hyperbolic groups are knapsack-semilinear} \label{sec-hyp-ksl}

In this section, we prove the following strengthening of Theorem~\ref{thm-ushakov}:

\begin{theorem} \label{thm-hyperbolic-semilinear}
Every hyperbolic group is knapsack-tame. 
\end{theorem}   
Let us  remark that the total number of vectors in a semilinear representation
can be exponential, even for the simplest case $G = \mathbb{Z}$. Take 
the (additively written) knapsack expression $E = x_1 + x_2 + \cdots + x_n - n$. 
Then $\Sol(E)$ is finite and consists of $\binom{2n-1}{n} \ge 2^n$ vectors.

Let us fix a $\delta$-hyperbolic group $G$ for the rest of Section~\ref{sec-hyp-ksl} and let $\Sigma$ be a finite symmetric generating set for $G$.

\subsection{Knapsack expressions of depth two}

We first consider knapsack expressions of depth $2$ where all powers are quasigeodesic.
It is well known that the semilinear sets are exactly the Parikh images of the regular languages.
We need a quantitative version of this result that was independently discovered by Kopczynski and Lin:

\begin{theorem}[c.f.~\mbox{\cite[Theorem 4.1]{To10}}, see also \cite{KopczynskiT10}] \label{thm-ko-lin}
Let $k$ be a fixed constant. Given an NFA $\mathcal{A}$ over an alphabet of size $k$ with $n$ states,
one can compute in polynomial time a semilinear representation of the Parikh image of $L(\mathcal{A})$.
Moreover, all numbers appearing in the semilinear representation are polynomially bounded in $n$ (in other words:
one can compute the semilinear representation with unary encoded numbers).
\end{theorem}

\begin{lemma} \label{lemma-k=2}
Let $\lambda$ and $\epsilon$ be fixed constants.
For all geodesic words $u_1, v_1, u_2, v_2 \in \Sigma^*$ such that $u_1 \neq \varepsilon \neq u_2$ and  
$u_1^n$, $u_2^n$ are $(\lambda, \epsilon)$-quasigeodesic for all $n \geq 0$,
the set $\{ (x_1, x_2) \in \N \times \N \mid v_1 u_1^{x_1} = u_2^{x_2} v_2 \text{ in } G\}$ is semilinear.
Moreover, one can compute a semi-linear representation whose magnitude is bounded by $p(|u_1| + |v_1| + |u_2| + |v_2|)$ for a 
fixed polynomial $p(n)$.
\end{lemma}  

\begin{proof}
Let $S := \{ (x_1, x_2) \in \N \times \N \mid v_1 u_1^{x_1} = u_2^{x_2} v_2 \text{ in } G\}$.
We will define an NFA $\mathcal{A}$ over the alphabet $\{a_1,a_2\}$ such that the Parikh image
of $L(A)$ is $S$.
Moreover, the number of states of $\mathcal{A}$ is polynomial in $|u_1|+|u_2|+|v_1|+|v_2|$. This allows us to apply Theorem~\ref{thm-ko-lin}.
We will allow transitions that are labelled with words (having length polynomial in $|u_1|+|u_2|+|v_1|+|v_2|$).
Moreover, instead of writing in the transitions these words, we write their Parikh images (so, for instance,
a transition $p \xrightarrow{a_1^2 a_2^3} q$ is written as $p \xrightarrow{(2,3)} q$.

Let $\ell_i = |u_i|$ and $m_i = |v_i|$.  Take the constant $\kappa$ from Lemma~\ref{lemma-rectangle} and
define $N_1 = \lambda (m_1 + 2 \delta + \kappa) + \epsilon$ and
$N_2 =  \lambda (m_2 + 2 \delta + \kappa) + \epsilon$.
We split the set $S$ into two parts:
\begin{itemize}
\item $S_1 = S \cap \{ (n_1, n_2) \in \N \times \N \mid  n_1 < (N_1+N_2)/\ell_1 \}$
\item $S_2 = S \cap \{ (n_1, n_2) \in \N \times \N \mid  n_1 \ge (N_1+N_2)/\ell_1 \}$
\end{itemize}
For all $(n_1,n_2) \in S_1$ we have $|u_1^{n_1}| = n_1 \ell_1 <  N_1+N_2$. Hence, $|\shlex(u_2^{n_2})| =
|\shlex(v_1 u_1^{n_1} v_2^{-1})| < N_1+N_2+m_1+m_2$.
Since $u_2^{n_2}$ is $(\lambda,\epsilon)$-quasigeodesic we get
$|u_2^{n_2}| = n_2 \ell_2 <  \lambda (N_1 + N_2 + m_1 + m_2) + \epsilon$, i.e., $n_2  <  (\lambda (N_1 + N_2 + m_1 + m_2) + \epsilon)/\ell_2$.
Hence, the set $S_1$ is finite  and has a semilinear representation where all numbers are bounded by 
$\mathcal{O}(m_1 + m_2)$.

We now deal with pairs $(n_1,n_2) \in S_2$, where $v_1 u_1^{n_1} = u_2^{n_2} v_2$ in $G$ and 
$n_1 \ge (N_1+N_2)/\ell_1$, i.e., $|u_1^{n_1}| \ge  N_1+N_2$. Consider such a pair $(n_1, n_2)$
and the quasigeodesic rectangle consisting of the four paths
$Q_1 = P[v_1]$,  $P_1 = v_1 \cdot P[u_1^{n_1}]$, $P_2 = P[u_2^{n_2}]$, and
$Q_2 = u_2^{n_2} \cdot P[v_2]$.
We factorize the word $u_1^{n_1}$ as $u_1^{n_1} = x y z$ with $|x| =  N_1$
and $|z| = N_2$. 
By Lemma~\ref{lemma-rectangle} we can factorize $u_2^{n_2}$ as $u_2^{n_2} = x' y' z'$ such that
there exist $c,d \in \mathcal{B}_{2\delta+2\kappa}(1)$ with $v_1 x = x' c$ and 
$d z = z' v_2$ in $G$, see Figure~\ref{fig-run} (where $n_1 = 20$, $n_2 = 10$, $\ell_1 = 2$ and $\ell_2 = 4$). Since
$u_2^{n_2}$ is $(\lambda,\epsilon)$-quasigeodesic, we have
\begin{eqnarray}
|x'| & \le & \lambda(m_1 + |x| + 2\delta+2\kappa) + \epsilon = \lambda(m_1 + N_1  + 2\delta+2\kappa) + \epsilon, \label{length-x'} \\
|z'|  & \le & \lambda(m_2 + |z| + 2\delta+2\kappa) + \epsilon =  \lambda(m_2 + N_2 + 2\delta+2\kappa) + \epsilon \label{length-z'}.
\end{eqnarray}
Consider now the subpath $P'_1$ of $P_1$ from $P_1(|x|)$ to $P_1(n_1 \ell_1 - |z|)$ and the subpath $P'_2$ of $P_2$
from $P_2(|x'|)$ to $P_2(n_2 \ell_2 - |z'|)$. These are the paths labelled with $y$ and $y'$, respectively, in 
Figure~\ref{fig-run}.
By Lemma~\ref{lemma-asynch-fellow-travel} these paths
asynchronously $\gamma$-fellow travel, where $\gamma := K(\delta, \lambda,\epsilon, 2\delta+2\kappa)$ is a constant. 
In Figure~\ref{fig-run} this is visualized by the part between the $c$-labelled edge and the $d$-labelled edge.
W.l.o.g. we assume that $\gamma  \geq 2\delta+2\kappa$. 

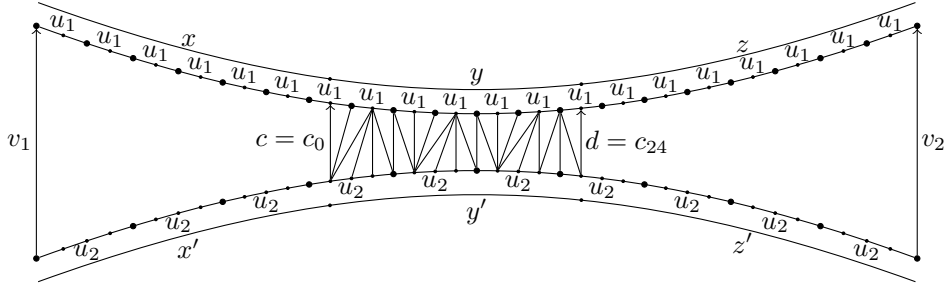
\begin{figure}[t]
 \centering{
 \scalebox{1}{
\begin{tikzpicture}
  \tikzstyle{small} = [circle,draw=black,fill=black,inner sep=.25mm]
  \tikzstyle{tiny} = [circle,draw=black,fill=black,inner sep=.1mm]
  \tikzstyle{zero} = [circle,inner sep=0mm]

    \node[small] (1) {} ;
    \node[small,  right = 11.5cm of 1] (2) {};
    \node[small,  above = 3cm of 1] (3) {};
    \node[small,  right = 11.5cm of 3] (4) {};
    
    \node[zero,  below = 2.5mm of 1] (1') {} ;
    \node[zero,  below = 2.5mm of 2] (2') {} ;
    \node[zero,  above = 2.5mm of 3] (3') {} ;
    \node[zero,  above = 2.5mm of 4] (4') {} ;
    
      \draw [->] (1) to  node[pos=.5,left=-.7mm] {$v_1$} (3);
      \draw [->] (2) to  node[pos=.5,right=-.7mm] {$v_2$} (4);
      
       \draw [-] (1') to [out=20, in=160] node[pos=.1625,below=-.7mm] {$x'$} node[pos=.325, tiny] (a') {}  node[pos=.625, tiny] (a'') {}   node[pos=.5,below=-.7mm] {$y'$}  node[pos=.8125,below=-.7mm] {$z'$}(2');
       \draw [-] (3') to [out=-20, in=-160] node[pos=.1625,above=-.7mm] {$x$} node[pos=.325, tiny] (b') {} node[pos=.625, tiny] (b'') {}    node[pos=.5,above=-.7mm] {$y$} node[pos=.8125,above=-.7mm] {$z$} (4');
       
      \draw [-] (1) to [out=20, in=160] 
      node[pos=.05,below=-.7mm] {$u_2$} 
      node[pos=.025, tiny] (a0) {}  
      node[pos=.05, tiny] (a1) {}  
      node[pos=.075, tiny] (a2) {}  
      node[pos=.1, small] (a3) {}  
      node[pos=.15,below=-.7mm]  {$u_2$}
      node[pos=.125, tiny] (a4) {}  
      node[pos=.15, tiny] (a5) {}  
      node[pos=.175, tiny] (a6) {}  
      node[pos=.2, small] (a7) {}  
      node[pos=.25,below=-.7mm]  {$u_2$}
      node[pos=.225, tiny] (a8) {}  
      node[pos=.25, tiny] (a9) {}  
      node[pos=.275, tiny] (a10) {}  
      node[pos=.3, small] (a11) {}     
      node[pos=.35,below=-.7mm]  {$u_2$}
      node[pos=.325, tiny] (a12) {}  
      node[pos=.35, tiny] (a13) {}  
      node[pos=.375, tiny] (a14) {}  
      node[pos=.4, small] (a15) {} 
      node[pos=.45,below=-.7mm]  {$u_2$}
      node[pos=.425, tiny] (a16) {}  
      node[pos=.45, tiny] (a17) {}  
      node[pos=.475, tiny] (a18) {}  
      node[pos=.5, small] (a19) {}  
      node[pos=.55,below=-.7mm]  {$u_2$}
      node[pos=.525, tiny] (a20) {}  
      node[pos=.55, tiny] (a21) {}  
      node[pos=.575, tiny] (a22) {}  
      node[pos=.6, small] (a23) {}  
      node[pos=.65,below=-.7mm]  {$u_2$}
      node[pos=.625, tiny] (a24) {}  
      node[pos=.65, tiny] (a25) {}  
      node[pos=.675, tiny] (a26) {}  
      node[pos=.7, small] (a27) {}  
      node[pos=.75,below=-.7mm]  {$u_2$}
      node[pos=.725, tiny] (a28) {}  
      node[pos=.75, tiny] (a29) {}  
      node[pos=.775, tiny] (a30) {}  
      node[pos=.8, small] (a31) {}  
      node[pos=.85,below=-.7mm]  {$u_2$}
      node[pos=.825, tiny] (a32) {}  
      node[pos=.85, tiny] (a33) {}  
      node[pos=.875, tiny] (a34) {}  
      node[pos=.9, small] (a35) {}  
      node[pos=.95,below=-.7mm]  {$u_2$}
      node[pos=.925, tiny] (a36) {}  
      node[pos=.95, tiny] (a37) {}  
      node[pos=.975, tiny] (a38) {}  
        (2);
       \draw [-] (3) to [out=-20, in=-160] 
      node[pos=.025,above=-.7mm] {$u_1$} 
      node[pos=.025, tiny] (b0) {}  
      node[pos=.05, small] (b1) {}  
      node[pos=.075,above=-.7mm]  {$u_1$}
       node[pos=.075, tiny] (b2) {}
      node[pos=.1, small] (b3) {}  
      node[pos=.125,above=-.7mm]  {$u_1$}
      node[pos=.125, tiny] (b4) {}  
      node[pos=.15, small] (b5) {}  
      node[pos=.175,above=-.7mm]  {$u_1$}
      node[pos=.175, tiny] (b6) {}  
      node[pos=.2, small] (b7) {}  
      node[pos=.225,above=-.7mm]  {$u_1$}
      node[pos=.225, tiny] (b8) {}  
      node[pos=.25, small] (b9) {}  
      node[pos=.275,above=-.7mm]  {$u_1$}
      node[pos=.275, tiny] (b10) {}  
      node[pos=.3, small] (b11) {}  
      node[pos=.325,above=-.7mm]  {$u_1$}
      node[pos=.325, tiny] (b12) {}  
      node[pos=.35, small] (b13) {}  
      node[pos=.375,above=-.7mm]  {$u_1$}
      node[pos=.375, tiny] (b14) {}  
      node[pos=.4, small] (b15) {}  
      node[pos=.425,above=-.7mm]  {$u_1$}
      node[pos=.425, tiny] (b16) {}  
      node[pos=.45, small] (b17) {}  
      node[pos=.475,above=-.7mm]  {$u_1$}
      node[pos=.475, tiny] (b18) {}  
      node[pos=.5, small] (b19) {}  
      node[pos=.525,above=-.7mm] {$u_1$} 
      node[pos=.525, tiny] (b20) {}
      node[pos=.55, small] (b21) {}  
      node[pos=.575,above=-.7mm]  {$u_1$}
      node[pos=.575, tiny] (b22) {}  
      node[pos=.6, small] (b23) {}  
      node[pos=.625,above=-.7mm]  {$u_1$}
      node[pos=.625, tiny] (b24) {}  
      node[pos=.65, small] (b25) {}  
      node[pos=.675,above=-.7mm]  {$u_1$}
      node[pos=.675, tiny] (b26) {} 
      node[pos=.7, small] (b27) {}  
      node[pos=.725,above=-.7mm]  {$u_1$}
      node[pos=.725, tiny] (b28) {}  
      node[pos=.75, small] (b29) {}  
      node[pos=.775,above=-.7mm]  {$u_1$}
      node[pos=.775, tiny] (b30) {}  
      node[pos=.8, small] (b31) {}  
      node[pos=.825,above=-.7mm]  {$u_1$}
      node[pos=.825, tiny] (b32) {}  
      node[pos=.85, small] (b33) {}  
      node[pos=.875,above=-.7mm]  {$u_1$}
      node[pos=.875, tiny] (b34) {}  
      node[pos=.9, small] (b35) {}  
      node[pos=.925,above=-.7mm]  {$u_1$}
      node[pos=.925, tiny] (b36) {}  
      node[pos=.95, small] (b37) {}  
      node[pos=.975,above=-.7mm]  {$u_1$} 
      node[pos=.975, tiny] (b38) {}
             (4); 
             
      \draw [->] (a12) to  node[pos=.5,left=-.7mm] {$c=c_0$} (b12);
      \draw [-] (a12) to  (b13);
      \draw [-] (a12) to  (b14);
      \draw [-] (a13) to  (b14);
      \draw [-] (a14) to  (b14);
      \draw [-] (a15) to  (b14);
      \draw [-] (a15) to  (b15);
      \draw [-] (a16) to  (b15);
      \draw [-] (a16) to  (b16);
      \draw [-] (a16) to  (b17);
      \draw [-] (a16) to  (b18);
      \draw [-] (a17) to  (b18);
      \draw [-] (a18) to  (b18);
      \draw [-] (a19) to  (b18);
      \draw [-] (a19) to  (b19);
      \draw [-] (a20) to  (b19);
      \draw [-] (a20) to  (b20);
      \draw [-] (a20) to  (b21);
       \draw [-] (a20) to  (b22);
       \draw [-] (a21) to  (b22);
       \draw [-] (a22) to  (b22);
       \draw [-] (a22) to  (b23);
        \draw [-] (a23) to  (b23);
         \draw [-] (a24) to  (b23);
      \draw [->] (a24) to node[pos=.5,right=-.7mm] {$d=c_{24}$} (b24);
   \end{tikzpicture}}}
\caption{\label{fig-run} Example for the construction from the proof of Lemma~\ref{lemma-k=2}.}
  \end{figure}

We now define the NFA $\mathcal{A}$ over the alphabet $\{a_1, a_2\}$ (recall the we replace edge labels from $\{a_1,a_2\}^*$ by their Parikh images). The state set of $\mathcal{A}$
is
\[
Q = \{ q_0, q_f \} \cup \{ (i,b,j) \mid 0 \le i < \ell_1, 0 \le j < \ell_2, b \in \mathcal{B}_\gamma(1)  \} .
\]
The  unique initial state is $q_0$ and the unique final state
is $q_f$.
To define the transitions of $\mathcal{A}$ set
$p = \lfloor N_1 / \ell_1 \rfloor = \lfloor|x|/|u_1|\rfloor$, $r = N_1 \bmod \ell_1 = |x| \bmod |u_1|$,
$s = \lceil N_2 / \ell_1 \rceil = \lceil |z| / |u_1| \rceil$, $t = -N_2 \bmod \ell_1 = -|z| \bmod |u_1|$.
Thus, we have $x = u_1^p u_1[:r]$ and $z = u_1^s[t+1:]$.
There are the following types of transitions (transitions without a label are implicitly labelled by the zero vector $(0,0)$),
where $0 \leq i < \ell_1$, $0 \le j < \ell_2$,  $b,b' \in \mathcal{B}_{\gamma}(1)$.
\begin{enumerate}
\item $q_0 \xrightarrow{(p,p')} (r, c, r')$ if there exists 
a number $0 \le k \le \lambda(m_1 + N_1  + 2\delta+2\kappa) + \epsilon$ (this is the possible range for the length of $x'$ in \eqref{length-x'}) 
such that $p' = \lfloor k/\ell_2\rfloor$, $r' = k \bmod \ell_2$, and 
$v_1 u_1^p u_1[:r] = u_2^{p'} u_2[:r'] c$ in $G$.
\item $(i, b, j) \xrightarrow{} (i+1, b', j)$ if $i+1 < \ell_1$ and $b u_1[i+1] = b'$ in $G$.
\item $(\ell_1-1, b, j) \xrightarrow{(1,0)} (0, b', j)$ if $b u_1[\ell_1] = b'$ in $G$.
 \item $(i, b, j) \xrightarrow{} (i, b', j+1)$ if $j+1 < \ell_2$ and $b =  u_2[j+1]b'$ in $G$.
\item $(i,b,\ell_2-1) \xrightarrow{(0,1)} (i, b', 0)$ if $b =  u_2[\ell_2]b'$ in $G$.
\item $(t, d, t')   \xrightarrow{(s,s')} q_f$ if there exists 
a number $0 \le k \le \lambda(m_2 + N_2 + 2\delta+2\kappa) + \epsilon$ (this is the possible range for the length of $z'$ in \eqref{length-z'}) 
such that $s' = \lceil k/\ell_2\rceil$, $t' = -k \bmod \ell_2$, and 
$d u_1[t+1:] u_1^s = u_2[t'+1:] u_2^{s'}  v_2$ in $G$.
\end{enumerate}
The construction is best explained using the example in Figure~\ref{fig-run}. As mentioned above, the vertical lines between $c = c_0$ and $d = c_{24}$ 
represent the asynchronous $\gamma$-fellow travelling. The vertical lines are labelled with group elements $c_0, c_1, \ldots, c_{23}, c_{24} \in \mathcal{B}_\gamma(1)$
from left to right. In order to not overload the figure we only show $c_0$ and $c_{24}$. Note that 
$x = u_1^6 u_1[1]$, $x' = u_2^3 u_2[1]$, $z = u_1^8[2:]$, $z' = u_2^4[2:]$.
Basically, the NFA $\mathcal{A}$ moves the vertical edges from left to right and thereby stores (i) the label $c_i$ of the vertical edge,
(ii) the position in the current $u_2$-factor where the vertical edge starts (position $0$ means that we have just completed a $u_2$-factor),
and (iii) the position in the current $u_1$-factor where the vertical edge ends. If a $u_1$-factor (resp., $u_2$-factor) is completed then
the automaton makes a $(1,0)$-labelled (resp., $(0,1)$-labelled) transition.
The automaton run corresponding to Figure~\ref{fig-run} is:
\begin{align*}
 q_0 \xrightarrow{(6,3)} & (1,c_0,1) \xrightarrow{(1,0)}  (0,c_1,1) \to (1,c_2,1) \to (1,c_3,2) \to (1,c_4,3) \xrightarrow{(0,1)} \\
 & (1, c_5,0) \xrightarrow{(1,0)} (0,c_6,0) \to (0,c_7,1) \to (1, c_8,1)  \xrightarrow{(1,0)} (0,c_9,1) \to \\
 & (1,c_{10},1) \to (1,c_{11},2) \to (1,c_{12},3) \xrightarrow{(0,1)} (1, c_{13}, 0) \xrightarrow{(1,0)} (0,c_{14},0) \to \\
 & (0, c_{15},1) \to (1, c_{16}, 1) \xrightarrow{(1,0)} (0, c_{17}, 1) \to (1, c_{18}, 1) \to (1, c_{19}, 2) \to \\
 & (1, c_{20}, 3) \xrightarrow{(1,0)} (0, c_{21}, 3) 
 \xrightarrow{(0,1)} (0, c_{22}, 0) \to (0, c_{23},1) \to (1, c_{24},1) \xrightarrow{(8,4)} q_f
\end{align*}
With the above intuition it is straightforward to show that the Parikh image of $L(\mathcal A)$ is indeed $S_2$.
Also note that the number of states of $\mathcal{A}$ is bounded by $\mathcal{O}(\ell_1 \ell_2)$.
The statement of the lemma then follows directly from Theorem~\ref{thm-ko-lin}.
\end{proof}

\subsection{Reduction to quasi-geodesic knapsack expressions}

Let us call a knapsack expression $E =  u_1^{x_1} v_1 u_2^{x_2} v_2 \cdots u_{k}^{x_{k}} v_{k}$ over $G$
$(\lambda,\epsilon)$-quasi\-geo\-desic if  all words $u_1, \ldots, u_k, v_1, \ldots, v_{k}$ are geodesic and for all $1 \le i \le k$ and all $n \geq 0$ the word
$u_i^n$ is $(\lambda,\epsilon)$-quasigeodesic.  We say that $E$ has {\em infinite order}, if all $u_i$ represent group elements of infinite order.
The goal of this section is to reduce a knapsack expression to a finite number 
(in fact, exponentially many) of $(\lambda,\epsilon)$-quasigeodesic 
knapsack expressions  of infinite order for certain constants $\lambda,\epsilon$:

\begin{proposition} \label{prop-geodesic-knapsack}
There exist fixed constants $\lambda,\epsilon$ such that
from a given knapsack expression $E$ over $G$ one can compute a finite list of  knapsack 
expressions $E_i$ ($i \in I$) over $G$ such that 
$$\Sol(E) = \bigcup_{i \in I} \big( (m_i \cdot \Sol(E_i) + d_i) \oplus  \mathcal{F}_i \big),$$
where the following additional properties hold:
\begin{itemize}
\item every $\mathcal{F}_i$ is a semilinear subset of $\N^Y$ for a subset $Y \subseteq X_E$,
\item the magnitude of every $\mathcal{F}_i$ is bounded by a constant that only depends on  $G$,
\item every $E_i$ is a $(\lambda,\epsilon)$-quasigeodesic knapsack expression of infinite order with variables from $Z := X_E \setminus Y$,
\item the size of every $E_i$ is bounded by $\mathcal{O}(|E|)$, and
\item all $m_i$ and $d_i$ are vectors from $\N^{Z}$ where all entries are bounded by a constant that only depends on $G$
(here, $m_i \cdot \Sol(E_i) = \{ m_i \cdot z \mid z \in \Sol(E) \}$ and $m_i \cdot z$ is the pointwise multiplication of the vectors $m_i$ and $z$).
\end{itemize}
\end{proposition}
Once Proposition~\ref{prop-geodesic-knapsack} is shown, we can conclude the proof of 
Theorem~\ref{thm-hyperbolic-semilinear} by showing
that all sets $\Sol(E_i)$ are semilinear and that their magnitudes are bounded by
$p(|E_i|)$ for a fixed polynomial $p(n)$. This will be achieved in the next section. 

For the proof of Proposition~\ref{prop-geodesic-knapsack} we mainly build on results from \cite{EpsteinH06}.
We fix the constants $L = 34\delta+2$ 
and $K = |\mathcal{B}_{4\delta}(1)|^2$. 

\begin{lemma}[c.f.~\mbox{\cite[Lemma~3.1]{EpsteinH06}}] \label{lemma-Ep-Ho1}
Let $u = u_1 u_2$ be shortlex reduced, where $|u_1| \le |u_2| \le |u_1|+1$.
Let $\tilde u = \shlex(u_2 u_1)$.
If $|\tilde u| \ge 2L+1$ then for every $n \geq 0$, the word $\tilde{u}^n$ is $L$-local $(1,2\delta)$-quasigeodesic.
\end{lemma}

The following lemma is not stated explicitly in \cite{EpsteinH06} but is shown in Section~3.2
(where the main argument is attributed to Delzant).

\begin{lemma}[c.f.~\cite{EpsteinH06}] \label{lemma-Ep-Ho2}
Let $u$ be geodesic such that $|u| \ge 2L+1$ and for every $n \geq 0$, the word $u^n$ is $L$-local $(1,2\delta)$-quasigeodesic.
Then one can compute $c \in \mathcal{B}_{4\delta}(1)$ and an integer $1 \leq m \leq K$ such that $(\shlex(c^{-1} u^m c))^n$
is geodesic for all $n \geq 0$.
\end{lemma}

\medskip
\noindent
{\em Proof of Proposition~\ref{prop-geodesic-knapsack}.}
We set $\lambda = N (2L+1)$ and $\epsilon = 2N^2 (2L+1)^2  + 2N (2L+1)$, where 
$N = |\mathcal{B}_{2\delta}(1)|$.
Consider a knapsack expression $E = u_1^{x_1} v_1 u_2^{x_2} v_2 \cdots u_{k}^{x_{k}} v_k$.
We can assume that every $u_i$ is shortlex reduced.
Let $g_i \in G$ be the group element represented by the word $u_i$. 

\medskip
\noindent
{\em Step 1.} In this first step we show how to reduce to the case where all $g_i$
have infinite order.
In a hyperbolic group $G$ the order of torsion elements is bounded by a fixed constant 
that only depends on $G$, see also the proof of \cite[Theorem~6.7]{MyNiUs14}).
This allows to check for each $g_i$ whether it has finite order, and to compute the order
in the positive case. Let $Y \subseteq \{ x_1, \ldots, x_k\}$ be those variables $x_i$ such that
$g_i$ has finite order. For $x_i \in Y$ let $o_i < \infty$ be the order of $g_i$.
Let $\mathcal{F}$ be the set of mappings $f : Y \to \mathbb{N}$ such that $0 \leq f(x_i) < o_i$
for all $x_i \in Y$. For every such mapping $f \in \mathcal{F}$ let 
$E_f$ be the knapsack expression that is obtained from $E$ by replacing
for every $x_i \in Y$ the power $u_i^{x_i}$ by $u_i^{f(x_i)}$ (which is merged with the word $v_i$). 
Moreover, let $\mathcal{F}_f$ be the set of all mappings $g : Y \to \mathbb{N}$  such that
$g(x_i) \equiv f(x_i) \bmod o_i$ for every $x_i \in Y$. Then
the set $\Sol(E)$ can be written as
$$
\Sol(E) = \bigcup_{f \in \mathcal{F}} \Sol(E_f) \oplus  \mathcal{F}_f .
$$
Note that $\mathcal{F}_f$ is a semilinar set of magnitude  $\mathcal{O}(1)$.

\medskip
\noindent
{\em Step 2.} We now consider a knapsack expression from $\mathcal{F}_f$. To simplify
notation, we denote this expression again with 
$E = u_1^{x_1} v_1 u_2^{x_2} v_2 \cdots u_{k}^{x_{k}} v_k$.
For every $i$, the group element $g_i$ represented by $u_i$ has infinite order.
We factorize $u_i$ uniquely as $u_i = u_{i,1} u_{i,2}$ where $|u_{i,1}| \le |u_{i,2}| \le |u_{i,1}|+1$, and let
$\tilde{u}_i = \shlex(u_{i,2}u_{i,1})$. Note that $|\tilde{u}_i| \le |u_i|$.
Let $\tilde g_i$ be the group element represented by
$\tilde{u}_i$. Since $\tilde g_i$ is conjugated
to $g_i$, also $\tilde g_i$ has infinite order.  By Lemma~\ref{lemma-cyclic-words-quasi-geo},
for every $n \geq 0$, the word $\tilde{u}_i^n$ is $(\lambda_i,\epsilon_i)$-quasigeodesic
for  $\lambda_i = N |\tilde{u}_i|$, $\epsilon_i = 2N^2 |\tilde{u}_i|^2  + 2N |\tilde{u}_i|$.
If $|\tilde{u}_i| < 2L+1$ then $\tilde{u}_i^n$ is 
$(\lambda, \epsilon)$-quasigeodesic for the constants $\lambda$ and $\epsilon$ defined
at the beginning of the proof.
We then replace $u_i^{x_i}$ by $u_{i,1} \tilde{u}_i^{x_i} u_{i,1}^{-1}$.
Note that for every $n \geq 0$, $u_{i,1} \tilde{u}_i^n u_{i,1}^{-1} = u_{i,1} (u_{i,2} u_{i,1})^n u_{i,1}^{-1} = 
(u_{i,1} u_{i,2})^n = u_i^n$ in $G$.

Now assume that $|\tilde{u}_i| \ge 2L+1$.
By Lemma~\ref{lemma-Ep-Ho1}, $\tilde{u}_i^n$ is $L$-local $(1,2\delta)$-quasi\-geo\-desic for every $n \geq 0$.
By Lemma~\ref{lemma-Ep-Ho2}, one can compute
$c_i \in \mathcal{B}_{4\delta}(1)$ and an integer $1 \leq m_i \leq K$ such that $(\shlex(c_i^{-1} \tilde{u}_i^{m_i} c_i))^n$
is geodesic (and hence $(1,0)$-quasigeodesic) for all $n \geq 0$. 
We then produce for every number $0 \leq d_i \leq m_i-1$ a new knapsack instance
by replacing $u_i^{x_i}$ by  $u_{i,1} \tilde{u}_i^{d_i} c_i (\shlex(c_i^{-1} \tilde{u}_i^{m_i} c_i))^{x_i} c_i^{-1} u_{i,1}^{-1}$. 
To make the description of the resulting knapsack expression more uniform we set $m_i = 1$ and $c_i = 1$
in case $|\tilde{u}_i| < 2L+1$. Then, the replacement of $u_i^{x_i}$ by $u_{i,1} \tilde{u}_i^{x_i} u_{i,1}^{-1}$
in case $|\tilde{u}_i| < 2L+1$ is the same as the one for the case $|\tilde{u}_i| \ge 2L+1$.
Let $m :  \{x_1, \ldots, x_k\} \to \N$ be the mapping with $m(x_i) = m_i$.

From the above discussion,
we obtain a finite set of 
$(\lambda,\epsilon)$-quasigeodesic knapsack 
expressions $E_d$ that are parameterized by a mapping $d : \{x_1, \ldots, x_k\} \to \N$ with
$0 \leq d(x_i) < m_i$ for all $1 \le i \le k$.
Let $\mathcal{D}$ be the set of all such mappings. We then have
$$
\Sol(E) = \bigcup_{d \in \mathcal{D}} (m \cdot \Sol(E_d) + d).
$$
Note that the magnitude of every $E_d$ is bounded linearly in the magnitude of $E$.

Finally, the statement of the proposition is directly obtained by combining the above steps 1 and 2.
\qed

\subsection{Proof of Theorem~\ref{thm-hyperbolic-semilinear}}

We now come to the proof of Theorem~\ref{thm-hyperbolic-semilinear}.
Consider a knapsack expression $E =  u_1^{x_1} v_1 u_2^{x_2} v_2 \cdots u_{k}^{x_{k}} v_k$.
We can assume that all $u_i, v_i$ are geodesic.
By Proposition~\ref{prop-geodesic-knapsack} we can moreover assume that for all $1 \leq i \leq k$,
$u_i$ represents a group element of infinite order and that $u_i^n$ is 
$(\lambda,\epsilon)$-quasigeodesic for all $n \ge 0$, where
$\lambda,\epsilon$ are fixed constants.
We want to show that $\Sol(E)$ is semilinear and has a magnitude that is polynomially
bounded by $|E|$.

For the case $k=1$ we have to consider all natural numbers
$n$ with $u_1^n = v_1^{-1}$ in $G$. Since $u_1$ represents a group element of 
infinite order there is at most one such $n$. Moreover, since $u_i^n$ is 
$(\lambda,\epsilon)$-quasigeodesic, such an $n$ has to satisfy
$|u_1| \cdot n \le \lambda |v_1| + \epsilon$, which yields
a linear bound on $n$.  

For the case $k=2$ we can directly use Proposition~\ref{lemma-k=2}.
Now assume that $k \geq 3$. 
We want to show that the set $\Sol(E)$ is a semilinear subset of $\mathbb{N}^k$
(later we will consider the magnitude of $\Sol(E)$).
For this we construct a Presburger formula with free variables $x_1, \ldots, x_k$
that is equivalent to $E=1$. We do this by induction on the depth $k$. Therefore, we can use 
in our Presburger formula also knapsack equations of the form $F=1$, where $F$ has depth
at most $k-1$.

It suffices to construct a Presburger formula for $\Sol(E) \cap (\mathbb{N} \setminus \{0\})^k$.
Note that $E=1$ is equivalent to $\bigvee_{I \subseteq \{1,\ldots,k\}} (E_I=1 \wedge \bigwedge_{i \in I} x_i > 0)$,
where $E_I$ is obtained from $E$ by removing for every $i \not\in I$ the power $u_i^{x_i}$.

Consider a tuple $(n_1, \ldots, n_k) \in \Sol(E) \cap (\mathbb{N} \setminus \{0\})^k$
and the corresponding $2k$-gon that is defined by the $(\lambda,\epsilon)$-quasigeodesic paths
$P_i = (u_1^{n_1} v_1 \cdots u_{i-1}^{n_{i-1}} v_{i-1}) \cdot P[u_i^{n_i}]$ and the geodesic paths
$Q_i = (u_1^{n_1} v_1 \cdots u_{i}^{n_{i}}) \cdot P[v_i]$, see Figure~\ref{polygon} for the case $k=3$.
Since all paths $P_i$ and $Q_i$ are $(\lambda,\epsilon)$-quasigeodesic,
we can apply \cite[Lemma 6.4]{MyNiUs14}: Every side of the $2k$-gon is contained in the 
$h$-neighborhoods of the other sides, where $h = \xi+\xi\log(2k)$ for a constant $\xi$ that only depends
on the constants $\delta, \lambda, \varepsilon$.

\newlength{\R}\setlength{\R}{2.7cm}
\begin{figure}[t]
\centering
\begin{tikzpicture}
  [inner sep=.5mm,
  minicirc/.style={circle,draw=black,fill=black,thick}]

  \node (circ1) at ( 45:\R) [minicirc] {};
  \node (circ2) at (135:\R) [minicirc] {};
  \node (circ3) at (165:\R) [minicirc] {};
  \node (circ4) at (255:\R) [minicirc] {};
  \node (circ5) at (285:\R) [minicirc] {};
  \node (circ6) at (15:\R) [minicirc] {};
  \draw [thick] (circ1) to [out=225, in=-45] node[above=1mm] {$u_2^{n_2}$} 
                      (circ2) to [out=285, in=15] node[above=0.5mm,left=0.5mm] {$v_1$} 
                      (circ3) to [out=345, in=75] node[below=0.5mm,left=0.5mm] {$u_1^{n_1}$} 
                      (circ4) to [out=45, in=135] node[below=1mm] {$v_3$} 
                      (circ5) to [out=105, in=195] node[below=0.5mm,right=0.5mm] {$u_3^{n_3}$} 
                      (circ6) to [out=165, in=255] node[right=1mm] {$v_2$} (circ1);
\end{tikzpicture}
\caption{\label{polygon} The $2k$-gon for $k = 3$ from the proof of Theorem~\ref{thm-hyperbolic-semilinear}}
\end{figure}
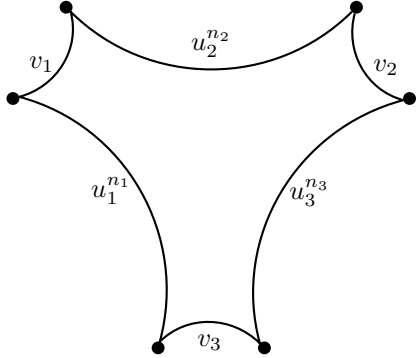

Let us now consider the side $P_2$ of the quasigeodesic $(2k)$-gon. It is labelled
with $u_2^{x_2}$. Its neighboring  sides are $Q_1$ and $Q_2$, which are labelled
with $v_1$ and $v_3$, respectively.
We distinguish several cases. In each case we cut the $2k$-gon into smaller 
pieces along paths of length $\leq 2h+1$ (length $h$ in some cases), and these smaller pieces will correspond to knapsack
expressions of depth $<k$. This is done until all knapsack expressions have depth at most two.
When we speak of a point on the $2k$-gon, we mean a node of the Cayley graph
(i.e., an element of the group $G$) and not a point in the interior of an edge.
Moreover, when we speak of the successor point of a point $p$, we refer to the clockwise order on the 
$2k$-gon, where the sides are traversed in the order $P_1,Q_1, \ldots, P_k,Q_k$.
We now distinguish the following cases:

\medskip
\noindent
{\em Case 1:} There is a point $p \in P_2$ that has distance at most $h$ from a point
$q$ that does not belong to $P_1 \cup Q_1 \cup Q_2 \cup P_3$. Thus $q$ must
belong to one of the paths $Q_3, P_4,  \ldots Q_{k-1},  P_k, Q_k$. Let $w$ be a geodesic
word of length at most $h$ that labels a path from $p$ to $q$.
There are two subcases:

\medskip
\noindent
{\em Case 1.1:} $q$ belongs to the paths $Q_i$, where $3 \le i \le k$.  
The situation is shown in Figure~\ref{polygon1.1}. 
We construct two new knapsack expressions $F_t$ and $G_t$ for all tuples $t = (w,u_{2,1},u_{2,2},v_{i,1},v_{i,2})$
such that $w \in \Sigma^*$ is of length at most $h$, $u_2 = u_{2,1} u_{2,2}$ and $v_i = v_{i,1} v_{i,2}$:
\begin{eqnarray*}
F_t &=& u_1^{x_1} v_1 u_2^{y_2} (u_{2,1} w v_{i,2}) u_{i+1}^{x_{i+1}} v_{i+1} \cdots u_k^{x_k} v_k \ \text{ and } \\
G_t  &=& u_{2,2} u_2^{z_2} v_2 u_3^{x_3} v_3 \cdots u_i^{x_i} (v_{i,1} w^{-1})
\end{eqnarray*}
Here $y_2$ and $z_2$ are new variables.
Note that $F_t$ and $G_t$ have depth at most $k-1$. 
Moreover, let $A_{1.1}$ be the following formula, where $t$ ranges over all tuples of the above form:
\[
A_{1.1} =  \bigvee_t \exists y_2,z_2: x_2 = y_2 + 1 + z_2 \wedge F_t = 1 \wedge G_t = 1
\]

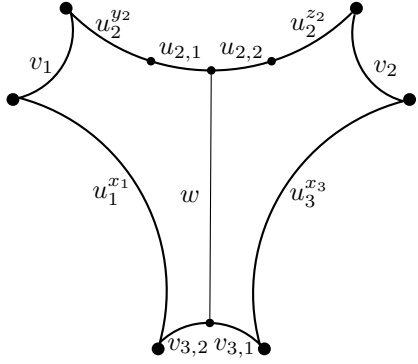
\begin{figure}[t]
\centering
\begin{tikzpicture}
  [inner sep=.5mm,
  minicirc/.style={circle,draw=black,fill=black,thick}]
  \tikzstyle{small} = [circle,draw=black,fill=black,inner sep=.3mm]
  \node (circ1) at ( 45:\R) [minicirc] {};
  \node (circ2) at (135:\R) [minicirc] {};
  \node (circ3) at (165:\R) [minicirc] {};
  \node (circ4) at (255:\R) [minicirc] {};
  \node (circ5) at (285:\R) [minicirc] {};
  \node (circ6) at (15:\R) [minicirc] {};
  \draw [thick] (circ1) to [out=225, in=-45] node[pos=.3,small] {} node[pos=.5,small]  (a) {} node[pos=.7,small] {} 
     node[above=0.2mm,pos=.4] {$u_{2,2}$} node[above=0.2mm,pos=.6] {$u_{2,1}$} node[pos=0.05,above=1mm,left=1mm] {$u_2^{z_2}$}
     node[pos=0.95,above=1mm,right=1mm] {$u_2^{y_2}$}
                      (circ2) to [out=285, in=15] node[above=0.5mm,left=0.5mm] {$v_1$} 
                      (circ3) to [out=345, in=75] node[below=0.5mm,left=0.5mm] {$u_1^{x_1}$} 
                      (circ4) to [out=45, in=135] node[below=1mm,pos=.3] {$v_{3,2}$} node[below=1mm,pos=.7] {$\; v_{3,1}$}  node[pos=.5,small]  (b){}
                      (circ5) to [out=105, in=195] node[below=0.5mm,right=0.5mm] {$u_3^{x_3}$} 
                      (circ6) to [out=165, in=255] node[right=1mm] {$v_2$} (circ1);
     \draw (a) to node[left=.8mm] {$w$} (b);
\end{tikzpicture}
\caption{\label{polygon1.1} Case~1.1 from the proof of Theorem~\ref{thm-hyperbolic-semilinear}}
\end{figure}

\noindent
{\em Case 1.2:} $q$ belongs to the path $P_i$, where $4 \le i \le k$
(this case can only occur if $k \ge 4$).
This case is analogous to Case 1.1. We only have to split $u_i^{x_i}$ as $u_i^{y_i} (u_{i,1} u_{i,2}) u_i^{z_i}$ 
(as we do for $u_2^{x_2}$).
We construct two new knapsack expressions $F_t$ and $G_t$ for all tuples $t = (w,u_{2,1},u_{2,2},u_{i,1},u_{i,2})$
such that $w \in \Sigma^*$ is of length at most $h$, $u_2 = u_{2,1} u_{2,2}$ and $u_i = u_{i,1} u_{i,2}$:
\begin{eqnarray*}
F_t &=& u_1^{x_1} v_1 u_2^{y_2} (u_{2,1} w u_{i,2}) u_i^{z_i} v_i  u_{i+1}^{x_{i+1}} v_{i+1} \cdots u_k^{x_k} v_k \ \text{ and } \\
G_t  &=& u_{2,2} u_2^{z_2} v_2 u_3^{x_3} v_3 \cdots u_{i-1}^{x_{i-1}} v_{i-1} u_i^{y_i} (u_{i,1} w^{-1})
\end{eqnarray*}
Here $y_2,z_2,y_i,z_i$ are new variables.
Note that $F_t$ and $G_t$ have depth at most $k-1$. 
Moreover, let $A_{1.2}$ be the following formula, where $t$ ranges over all tuples of the above form:
\[
A_{1.2} =  \bigvee_t \exists y_2,z_2,y_{i},z_{i} : x_2 = y_2 + 1 + z_2 \wedge x_{i}  = y_{i} + 1 + z_i \wedge F_t = 1 \wedge G_t = 1
\]
{\em Case 2:} Every point on $P_2$ that has distance at most $h$ from a point
on $P_1 \cup Q_1 \cup Q_2 \cup P_3$.

\medskip
\noindent
{\em Case 2.1:} The end point of $P_2$ (i.e., the point connecting $P_2$ with $Q_2$) has distance at most $h$
from a point on $Q_1$, see Figure~\ref{polygon2.1}.
For all tuples $t = (w,v_{1,1},v_{1,2})$
such that $w \in \Sigma^*$ is of length at most $h$ and $v_1 = v_{1,1} v_{1,2}$
we construct two new knapsack expressions 
\begin{eqnarray*}
F_t =  u_2^{x_2} (w v_{1,2}) \ \text{ and } \
G_t  =  u_1^{x_1} (v_{1,1} w^{-1} v_2) u_3^{x_3} v_3 \cdots u_k^{x_k} v_k
\end{eqnarray*}
and the formula
\[
A_{2.1} =  \bigvee_t  F_t = 1 \wedge G_t = 1, 
\]
where $t$ ranges over all tuples of the above form.
Note that $F_t$  has depth one and $G_t$ has depth $k-1$.

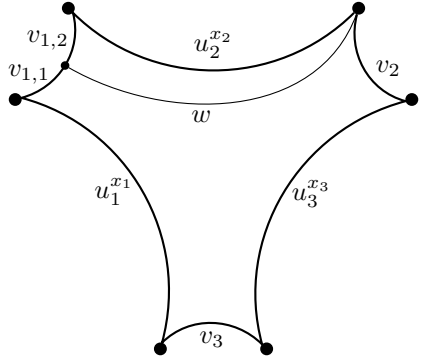
\begin{figure}[t]
\centering
\begin{tikzpicture}
  [inner sep=.5mm,
  minicirc/.style={circle,draw=black,fill=black,thick}]
  \tikzstyle{small} = [circle,draw=black,fill=black,inner sep=.3mm]
  \node (circ1) at ( 45:\R) [minicirc] {};
  \node (circ2) at (135:\R) [minicirc] {};
  \node (circ3) at (165:\R) [minicirc] {};
  \node (circ4) at (255:\R) [minicirc] {};
  \node (circ5) at (285:\R) [minicirc] {};
  \node (circ6) at (15:\R) [minicirc] {};
  \draw [thick] (circ1) to [out=225, in=-45]  node[above=1mm] {$u_2^{x_2}$} 
                      (circ2) to [out=285, in=15] node[left=.1mm,pos=.25] {$v_{1,2}$} node[above=.5mm,left=.5mm,pos=.6] {$v_{1,1}$}  node[pos=.5,small]  (b){}
                      (circ3) to [out=345, in=75] node[below=0.5mm,left=0.5mm] {$u_1^{x_1}$} 
                      (circ4) to [out=45, in=135] node[below=0.5mm] {$v_3$}
                      (circ5) to [out=105, in=195] node[below=0.5mm,right=0.5mm] {$u_3^{x_3}$} 
                      (circ6) to [out=165, in=255] node[right=1mm] {$v_2$}
                      (circ1);
     \draw (circ1) to [out=-110, in=-30] node[below=.3mm,pos=.6] {$w$} (b);
\end{tikzpicture}
\caption{\label{polygon2.1} Case~2.1 from the proof of Theorem~\ref{thm-hyperbolic-semilinear}}
\end{figure}

\medskip
\noindent
{\em Case 2.2:} The end point of $P_2$ (i.e., the point connecting $P_2$ with $Q_2$) has distance at most $h$
from a point on $P_1$, see Figure~\ref{polygon2.2}.
For all tuples $t = (w, u_{1,1},u_{1,2})$
such that $w \in \Sigma^*$ is of length at most $h$ and $u_1 = u_{1,1} u_{1,2}$,
we construct two new knapsack expressions
\begin{eqnarray*}
F_t =  u_1^{z_1} v_1 u_2^{x_2} (w u_{1,2}) \ \text{ and } \
G_t  =  u_1^{y_1} (u_{1,1} w^{-1} v_2) u_3^{x_3} v_3 \cdots u_k^{x_k} v_k
\end{eqnarray*}
and the formula
\[
A_{2.2} =  \bigvee_t \exists y_1,z_1 : x_1 = y_1 + 1 + z_1 \wedge F_t = 1 \wedge G_t = 1,
\]
where  $t$ ranges over all tuples of the above form.
Note that $F_t$  has depth two and $G_t$ has depth $k-1$.

\begin{figure}[t]
\centering
\begin{tikzpicture}
  [inner sep=.5mm,
  minicirc/.style={circle,draw=black,fill=black,thick}]
  \tikzstyle{small} = [circle,draw=black,fill=black,inner sep=.3mm]
  \node (circ1) at ( 45:\R) [minicirc] {};
  \node (circ2) at (135:\R) [minicirc] {};
  \node (circ3) at (165:\R) [minicirc] {};
  \node (circ4) at (255:\R) [minicirc] {};
  \node (circ5) at (285:\R) [minicirc] {};
  \node (circ6) at (15:\R) [minicirc] {};
  \draw [thick] (circ1) to [out=225, in=-45]  node[above=1mm] {$u_2^{x_2}$} 
                      (circ2) to [out=285, in=15] node[above=0.5mm,left=0.5mm] {$v_1$} 
                      (circ3) to [out=345, in=75] node[pos=.3,small] {} node[pos=.5,small]  (a) {} node[pos=.7,small] {} 
     node[left=-.3mm,pos=.62] {$u_{1,1}$} node[left=0mm,pos=.42] {$u_{1,2}$} node[pos=0.25,left=1mm] {$u_1^{z_1}$}
     node[pos=0.85,left=-.5mm] {$u_1^{y_1}$}
                      (circ4) to [out=45, in=135] node[below=0.5mm] {$v_3$}
                      (circ5) to [out=105, in=195] node[below=0.5mm,right=0.5mm] {$u_3^{x_3}$} 
                      (circ6) to [out=165, in=255] node[right=1mm] {$v_2$}
                      (circ1);
     \draw (circ1) to [out=-110, in=30] node[above=.3mm,pos=.6] {$w$} (a);
\end{tikzpicture}
\caption{\label{polygon2.2} Case~2.2 from the proof of Theorem~\ref{thm-hyperbolic-semilinear}}
\end{figure}

\medskip
\noindent
If on the other hand the end point of $P_2$ has distance $> h$ from all points on $P_1 \cup Q_1$, 
then there must be two points $p_1, p_2$  on $P_2$ such that $p_2$ is the successor point of $p_1$ when travelling
along $P_2$ (i.e., $d(p_1, p_2) = 1$), and $p_1$ has distance at most $h$ from a point $q_1 \in P_1 \cup Q_1$, while
$p_2$ has distance at most $h$ from a point on $q_2 \in Q_2 \cup P_3$. 
Thus, the distance between $q_1$ and $q_2$ is at most $2h+1$.
Let $w$ be a word that labels a geodesic path from $q_1$ to $q_2$ (thus, $|w| \le 2h+1$).
This leads to the following four subcases.

\medskip
\noindent
{\em Case 2.3:} $q_1 \in Q_1$ and $q_2 \in Q_2$, see Figure~\ref{polygon2.3}.
This case is very similar to Case~2.1. For every tuple $t = (w, v_{1,1}, v_{1,2}, v_{2,1}, v_{2,2})$
with $|w| \le 2h+1$, $v_1 = v_{1,1} v_{1,2}$ and $v_2 = v_{2,1} v_{2,2}$ 
we obtain two new knapsack expressions
\begin{eqnarray*}
F_t = F_t = v_{1,2} u_2^{x_2} (v_{2,1} w) \ \text{ and } \
G_t  =  u_1^{x_1} (v_{1,1} w^{-1} v_{2,2}) u_3^{x_3} v_3 \cdots u_k^{x_k} v_k 
\end{eqnarray*}
and the formula
\[
A_{2.3} = \bigvee_t  F_t = 1 \wedge G_t = 1,
\]
where $t$ ranges over all tuples of the above form.

\begin{figure}[t]
\centering
\begin{tikzpicture}
  [inner sep=.5mm,
  minicirc/.style={circle,draw=black,fill=black,thick}]
  \tikzstyle{small} = [circle,draw=black,fill=black,inner sep=.3mm]
  \node (circ1) at ( 45:\R) [minicirc] {};
  \node (circ2) at (135:\R) [minicirc] {};
  \node (circ3) at (165:\R) [minicirc] {};
  \node (circ4) at (255:\R) [minicirc] {};
  \node (circ5) at (285:\R) [minicirc] {};
  \node (circ6) at (15:\R) [minicirc] {};
  \draw [thick] (circ1) to [out=225, in=-45] node[above=1mm] {$u_2^{x_2}$} 
                      (circ2) to [out=285, in=15] node[left=.1mm,pos=.25] {$v_{1,2}$} node[above=.5mm,left=.5mm,pos=.6] {$v_{1,1}$}  node[pos=.5,small]  (b1){}
                      (circ3) to [out=345, in=75] node[below=0.5mm,left=0.5mm] {$u_1^{x_1}$} 
                      (circ4) to [out=45, in=135] node[below=0.5mm] {$v_3$}
                      (circ5) to [out=105, in=195] node[below=0.5mm,right=0.5mm] {$u_3^{x_3}$} 
                      (circ6) to [out=165, in=255] node[right=.1mm,pos=.35] {$v_{2,2}$} node[above=.5mm,right=0mm,pos=.75] {$v_{2,1}$}  node[pos=.5,small]  (b2){}
                      (circ1);
     \draw (b1) to [out=-30, in=210] node[below=.3mm,pos=.6] {$w$} (b2);
\end{tikzpicture}
\caption{\label{polygon2.3} Case~2.3 from the proof of Theorem~\ref{thm-hyperbolic-semilinear}}
\end{figure}

\medskip
\noindent
{\em Case 2.4:} $q_1 \in P_1$ and $q_2 \in Q_2$, see Figure~\ref{polygon2.4}.
This case is very similar to Case~2.2. 
For every tuple $t = (w, u_{1,1},u_{1,2}, v_{2,1}, v_{2,2})$
such that $|w| \le 2h+1$, $u_1 = u_{1,1} u_{1,2}$, and $v_2 = v_{2,1} v_{2,2}$ 
we obtain two new knapsack expressions
\begin{eqnarray*}
F_t = u_{1,2} u_1^{z_1} v_1 u_2^{x_2} (v_{2,1} w) \ \text{ and } \
G_t  =  u_1^{y_1} (u_{1,1} w^{-1} v_{2,2}) u_3^{x_3} v_3 \cdots u_k^{x_k} v_k 
\end{eqnarray*}
and the formula
\[
A_{2.4} =  \bigvee_t \exists y_1,z_1 : x_1 = y_1 + 1 + z_1 \wedge F_t = 1 \wedge G_t = 1,
\]
where  $t$ ranges over all tuples of the above form.

\begin{figure}[t]
\centering
\begin{tikzpicture}
  [inner sep=.5mm,
  minicirc/.style={circle,draw=black,fill=black,thick}]
  \tikzstyle{small} = [circle,draw=black,fill=black,inner sep=.3mm]
  \node (circ1) at ( 45:\R) [minicirc] {};
  \node (circ2) at (135:\R) [minicirc] {};
  \node (circ3) at (165:\R) [minicirc] {};
  \node (circ4) at (255:\R) [minicirc] {};
  \node (circ5) at (285:\R) [minicirc] {};
  \node (circ6) at (15:\R) [minicirc] {};
  \draw [thick] (circ1) to [out=225, in=-45] node[above=1mm] {$u_2^{x_2}$} 
                      (circ2) to [out=285, in=15] node[above=0.5mm,left=0.5mm] {$v_1$} 
                      (circ3) to [out=345, in=75] node[pos=.3,small] {} node[pos=.5,small]  (b1) {} node[pos=.7,small] {} 
     node[left=-.3mm,pos=.62] {$u_{1,1}$} node[left=0mm,pos=.42] {$u_{1,2}$} node[pos=0.25,left=1mm] {$u_1^{z_1}$}
     node[pos=0.85,left=-.5mm] {$u_1^{y_1}$}
                      (circ4) to [out=45, in=135] node[below=0.5mm] {$v_3$}
                      (circ5) to [out=105, in=195] node[below=0.5mm,right=0.5mm] {$u_3^{x_3}$} 
                      (circ6) to [out=165, in=255] node[right=.1mm,pos=.35] {$v_{2,2}$} node[above=.5mm,right=0mm,pos=.75] {$v_{2,1}$}  node[pos=.5,small]  (b2){}
                      (circ1);
     \draw (b1) to [out=30, in=210] node[below=.3mm,pos=.6] {$w$} (b2);
\end{tikzpicture}
\caption{\label{polygon2.4} Case~2.4 from the proof of Theorem~\ref{thm-hyperbolic-semilinear}}
\end{figure}

\medskip
\noindent
{\em Case 2.5:} $q_1 \in Q_1$ and $q_2 \in P_3$.
This case is analogous to Case~2.4.

\medskip
\noindent
{\em Case 2.6:} $q_1 \in P_1$ and $q_2 \in P_3$, see Figure~\ref{polygon2.6}.
For every tuple $$(w_1, w_2, w, u_{1,1}, u_{1,2}, u_{2,1}, u_{2,2}, u_{3,1}, u_{3,2})$$
such that $|w| \leq 2k+1$, $|w_1| \leq h$, $|w_2| \leq h+1$, $w = w_1^{-1} w_2$ in $G$,
$u_1 =  u_{1,1} u_{1,2}$, $u_2 = u_{2,1} u_{2,2}$, and $u_3 = u_{3,1} u_{3,2}$
we obtain three new knapsack expressions
\begin{eqnarray*}
F_t & = & u_1^{z_1} v_1 u_2^{y_2} (u_{2,1} w_1 u_{1,2}), \\
G_t  & = & u_2^{z_2} v_2 u_3^{y_3} (u_{3,1} w_2^{-1} u_{2,2}) \ \text{ and } \\
H_t & = &  u_3^{z_3} v_3 u_4^{x_4} v_4 \cdots u_k^{x_k} v_k u_1^{y_1} (u_{1,1} w u_{3,2}).
\end{eqnarray*}
and the formula
\[
A_{2.6} = \bigwedge_{t} \exists y_1, z_1, y_2, z_2, y_3, z_3 :  \bigwedge_{i=1}^3 x_i = y_i + 1 + z_i \wedge  F_t = 1 \wedge G_t = 1 \wedge H_t = 1,
\]
where $t$ ranges over all tuples of the above form.
Note that $F_t$ and $G_t$ have depth $2$ and that $H_t$ has depth $k-1$.
 
\begin{figure}[t]
\centering
\begin{tikzpicture}
  [inner sep=.5mm,
  minicirc/.style={circle,draw=black,fill=black,thick}]
  \tikzstyle{small} = [circle,draw=black,fill=black,inner sep=.3mm]
  \node (circ1) at ( 45:\R) [minicirc] {};
  \node (circ2) at (135:\R) [minicirc] {};
  \node (circ3) at (165:\R) [minicirc] {};
  \node (circ4) at (255:\R) [minicirc] {};
  \node (circ5) at (285:\R) [minicirc] {};
  \node (circ6) at (15:\R) [minicirc] {};
  \draw [thick] (circ1) to [out=225, in=-45] node[pos=.3,small] {} node[pos=.5,small]  (a) {} node[pos=.7,small] {} 
     node[above=0.2mm,pos=.4] {$u_{2,2}$} node[above=0.2mm,pos=.6] {$u_{2,1}$} node[pos=0.05,above=1mm,left=1mm] {$u_2^{z_2}$}
     node[pos=0.95,above=1mm,right=1mm] {$u_2^{y_2}$}
                      (circ2) to [out=285, in=15] node[above=0.5mm,left=0.5mm] {$v_1$} 
                      (circ3) to [out=345, in=75] node[pos=.3,small] {} node[pos=.5,small]  (b1) {} node[pos=.7,small] {} 
     node[left=-.3mm,pos=.62] {$u_{1,1}$} node[left=0mm,pos=.42] {$u_{1,2}$} node[pos=0.25,left=1mm] {$u_1^{z_1}$}
     node[pos=0.85,left=-.5mm] {$u_1^{y_1}$}
                      (circ4) to [out=45, in=135] node[below=0.5mm] {$v_3$}
                      (circ5) to [out=105, in=195] node[pos=.3,small] {} node[pos=.5,small]  (b2) {} node[pos=.7,small] {} 
     node[right=0.2mm,pos=.38] {$u_{3,2}$} node[right=0.5mm,pos=.58] {$u_{3,1}$} node[pos=0.15,right=0mm] {$u_3^{z_3}$}
     node[pos=0.78,below=1mm,right=1mm] {$u_3^{y_3}$}
                      (circ6) to [out=165, in=255] node[right=1mm] {$v_2$}
                      (circ1);
        \draw (a) to [out=-90, in=30] node[left=.5mm,pos=.45] {$w_1$} (b1);
        \draw (a) to [out=-90, in=150] node[right=.5mm,pos=.45] {$w_2$} (b2);
        \draw (b1) to [out=30, in=150] node[below=.5mm,pos=.5] {$w$} (b2);
\end{tikzpicture}
\caption{\label{polygon2.6} Case~2.6 from the proof of Theorem~\ref{thm-hyperbolic-semilinear}}
\end{figure}

\medskip
\noindent
 This concludes the construction of a Presburger formula for the set $\Sol(E)$ and shows the semilinearity of $\Sol(E)$.
 It remains to argue that the magnitude of $\Sol(E)$ is bounded polynomially in $|E|$. 
 Iterating the above splitting procedure results in an exponentially large disjunction of conjunctive formulas 
 of the form
 \begin{equation} \label{eq-final-conjunction}
\exists y_1, \ldots, y_m 
\bigwedge_{i \in I} E_i = 1 \bigwedge_{j \in J} z_j = z'_j + z''_j + 1
\end{equation}
 where every $E_i$ is a knapsack expression of depth at most two. Moreover, for $i \neq j$, $E_i$ and $E_j$ have no common variables.
 The existentially quantified variables $y_1, \ldots, y_m$ are the new variables that were introduced when splitting
 factors $u_i^{x_i}$ (e.g., $y_2, z_2$ in the formula $A_{1.1}$). The variables $z_j, z'_j, z''_j$ in  \eqref{eq-final-conjunction}
 are from $\{ x_1, \ldots, x_k, y_1, \ldots, y_m\}$.
 The equations $z_j = z'_j + z''_j + 1$ in \eqref{eq-final-conjunction}
 result from the splitting of factors $u_i^{x_i}$. For instance, $x_2 = y_2 +1 + z_2$ in $A_{1.1}$ is one 
 such equation. 
 
  In order to bound the magnitude of $\Sol(E)$ it suffices to consider a single conjunctive formula 
 of the form \eqref{eq-final-conjunction}, since disjunction corresponds to union of semilinear sets, which
 does not increase the magnitude. We can also ignore the existential quantifiers in \eqref{eq-final-conjunction},
 because existential quantification corresponds to projection onto some of the coordinates, which cannot increase
 the magnitude. 
 Hence, we have to consider the magnitude of the semilinear set $A$ defined
 by
  \begin{equation} \label{eq-final-conjunction'}
\bigwedge_{i \in I} E_i = 1 \bigwedge_{j \in J} z_j = z'_j + z''_j + 1 .
\end{equation}
 The splitting process that finally produces formula \eqref{eq-final-conjunction'} can be seen
 as a tree $T$, where every node $v$ is labelled with a knapsack expression $E(v)$, the root is labelled with $E$, the leaves
 are labelled with the expressions $E_i$ ($i \in I$) from \eqref{eq-final-conjunction} and the children of a node $v$ are labelled with 
 the expressions into which $E(v)$ is decomposed. The number of children of every node is at most three (three children are only produced
 in Case~2.6).

 Let us first show that the size of this tree $T$ is bounded by $\mathcal{O}(k^2)$. 
 We assign to each node $v$ of $T$ the number $d(v) :=$ depth of the knapsack expression $E(v)$.
 Note that $d(v) \leq 2$ if and only if $v$ is a leaf.  If $E(v)$ is split according to one of the Cases 2.1--2.6 then
 $v$ has $j \leq 3$ children $v_1, \ldots, v_j$, where $v_1, \ldots, v_{j-1}$ are leaves (their $d$-value is one or two) and 
 $d(v_j) = d(v)-1$. 
 If $E(v)$ is split according to Case~1.1 or 1.2 then
 $v$ has two children $v_1$ and $v_2$ such that (i) $d(v_1), d(v_2) < d(v)$, 
 (ii) $d(v_1), d(v_2) \ge 2$ and $d(v_1) + d(v_2) = d(v)+1$ in Case~1.1,
 and (iii) $d(v_1), d(v_2) \ge 3$ and $d(v_1) + d(v_2) = d(v)+2$ in Case~1.2. 
 Let $T'$ be the tree that is obtained by removing all leaves with $d$-value at most $2$.
 It suffices to show that the size of $T'$ is bounded by $\mathcal{O}(k^2)$. 
 All leaves of $T'$ have the $d$-value $3$. Moreover, every non-leaf $v$ of $T'$ has either exactly
 one child $v'$ with $d(v) > d(v')$ or two children $v_1$ and $v_2$ such that
 $d(v) \geq d(v_1) + d(v_2) - 2$. Let $n_0$ be the number of leaves of $T'$ and $n_2$ be the number of nodes of $T'$
 with exactly two children. 
 From the above equations, it follows that the root $r$ of $T'$
 satisfies $k = d(r) \ge 3 n_0 - 2 n_2$. Moreover, $n_2 = n_0 - 1$.
 We get $k \ge n_0 + 2$, i.e., $n_0 \le k-2$ and $n_2 \le k-3$.
 Since every path from the root
 to a leaf can contain at most $k$ nodes having a single child,
 we must have $n_1 \le  (k-2)k$.
 This shows that the size of $T'$ and hence of $T$ is bounded by
 $\mathcal{O}(k^2)$. Thus, we also have $|I| \leq \mathcal{O}(k^2)$ in   \eqref{eq-final-conjunction'}.
 
Next, we show that for every $i \in I$, $|E_i|$ is bounded polynomially in $|E|$. To see this, consider
 a single splitting step. In each of the above Cases~1.1--2.6 the argument is similar. Consider for instance
 Case 2.6, where the knapsack expression $E$ is replaced by three knapsack expressions $F_t, G_t, H_t$.
 We can bound the sizes of these expressions by $|F_t| \le |E| + |u_{1,2}| + |u_{2,1}| + |w_1| \le |E| + |u_1| + |u_2| + h$,
 $|G_t| \le |E| + |u_{2,2}| + |u_{3,1}| + |w_2| \le |E| + |u_{2}| + |u_{3}| + h+1$, and
 $H_t| \le |E| + |u_{1,1}| + |u_{3,2}| + |w| \le |E| + |u_{1}| + |u_{3}| + 2h+1$.
 The number of splitting steps
 that finally leads to an $E_i$ is bounded by $k$ (since the depth of the knapsack expressions
 is reduced in each step).
 Hence, the size of each knapsack expression $E_i$ in \eqref{eq-final-conjunction'} is bounded by
 $|E| + 2k|E| + k(2h+1) = (2k+1) |E| + k (2\xi+2\xi\log(2k)+1) \le \mathcal{O}(|E|^2)$.
 Since every $E_i$ has depth at most two, there is a fixed polynomial $p(n)$ such that
  the magnitude of every set $\Sol(E_i)$ is bounded by $p(|E|)$. Hence, also $\bigoplus_{i \in I} \Sol(E_i)$
  is a semilinear set of magnitude at most $p(|E|)$ (the $\oplus$-operator on semilinear sets does
  not increase the magnitude). Note that $\bigoplus_{i \in I} \Sol(E_i)$ is the semilinear set defined by
  the conjunction $\bigwedge_{i \in I} E_i = 1$.
    
  To bound the magnitude of the semilinear set $A$ defined by  \eqref{eq-final-conjunction'}, one has to consider
  also the additional equations $z_j = z'_j + z''_j + 1$ for $j \in J$. Let $U$ be the set of variables that appear
  in the  knapsack expressions $E_i$ ($i \in I$). 
  Note that the dimension of $\bigoplus_{i \in I} \Sol(E_i)$
  is $|U|$. Since every  knapsack expression $E_i$ ($i \in I$) contains at most
  two variables, we can bound the dimension of $\bigoplus_{i \in I} \Sol(E_i)$ by $2 |I| \leq \mathcal{O}(k^2)$.
  Note that for each equation $z_j = z'_j + z''_j + 1$ there exists a node $v$ in the tree $T$ with children $v', v''$ such
  that $z_j$ is a variable from $E(v)$, $z'_j$ is a variable from $E(v')$, and $z''_j$ is a variable from $E(v'')$.
  This implies that every variable $z_j$ is a sum of pairwise different variables from $U$ plus a constant that is bounded
  by $|T| \le \mathcal{O}(k^2)$. 
  Therefore  the magnitude of $A$ is bounded by $\mathcal{O}(k^2 \cdot p(|E|))$, which is polynomial in $|E|$.
  This concludes the proof.
 \qed

\section{More groups with knapsack in LogCFL}
\label{sec-more-logcfl}

Let $\mathcal{C}$ be the smallest class of groups such that (i) every hyperbolic group belongs to $\mathcal{C}$,
(ii) if $G \in \mathcal{C}$ then also $G \times \Z \in \mathcal{C}$, and (iii) if $G, H \in \mathcal{C}$ then also
$G * H \in \mathcal{C}$ (where $G * H$ is the free product of $G$ and $H$).
The class $\mathcal{C}$ contains groups that are not hyperbolic (e.g., $\mathbb{Z} \times \mathbb{Z}$).
From Theorem~\ref{thm-hyperbolic-semilinear} and Proposition~\ref{prop-knapsack-tame} we get:

\begin{proposition} \label{prop-knapsack-tame-C}
Every group from the class $\mathcal{C}$ is knapsack-tame and hence polynomially knapsack-bounded.
\end{proposition}
From Theorem~\ref{thm-hyp-OW-AuxPDA} and \ref{thm-free-directZ-OW-AuxPDA} we get:
\begin{proposition} \label{prop-OW-AuxPDA-C}
Every group from the class $\mathcal{C}$ belongs to OW-AuxPDA.
\end{proposition}
Proposition~\ref{prop-knapsack-tame-C} and \ref{prop-OW-AuxPDA-C} together with 
Theorem~\ref{thm-acyclic-NFA} and \ref{thm-membership-to-knapsack} yield:

\begin{corollary}  \label{coro-final}
For every group $G$ from the class $\mathcal{C}$, membership for acyclic NFAs over $G$ and knapsack for $G$ 
both belong to $\LogCFL$.
\end{corollary}
Corollary~\ref{coro-final} generalizes Corollaries~\ref{coro-acyclic-NFA-hyp} and \ref{thm-knapsack-logcfl}
as well as \cite[Corollary~22]{FrenkelNU15}, where it was shown that knapsack can be solved
in polynomial time for a free product of hyperbolic groups and finitely generated abelian groups.

\section{Conclusion} 

In this paper, it is shown that every hyperbolic group is knapsack-tame and that the knapsack problem 
can be solved in $\LogCFL$. Here is a list of open problems that one might consider for future work.
\begin{itemize}
\item For the following important groups, it is not known whether the knapsack problem
is decidable:  braid groups $B_n$ (with $n \geq 3$), solvable Baumslag-Solitar groups $\mathsf{BS}_{1,p} = \langle a,t \mid t^{-1} a t = a^p \rangle$
(with $p \geq 2$), and automatic groups which are not in any of the known classes with a decidable knapsack problem.
\item In \cite{KoenigLohreyZetzsche2015a}, it was shown that knapsack is decidable for every co-context-free group. The algorithm from \cite{KoenigLohreyZetzsche2015a} has 
an exponential running time. Is there a more efficient solution?
\item  Is there a polynomially knapsack-bounded group which is not knapsack-tame?
\end{itemize}

%\bibliography{bib}

\end{document}